\UseRawInputEncoding
\documentclass[12pt,reqno]{article}
\usepackage{xcolor}
\usepackage{amssymb}
\usepackage{comment}
\usepackage{amsmath}
 \usepackage{mathrsfs}
\usepackage{amsfonts}
\usepackage[colorlinks,linkcolor=red,anchorcolor=green,citecolor=blue]{hyperref}
\usepackage{graphicx,cite,enumerate,tikz,amsthm,amscd,bm,amssymb,changepage,mathrsfs}
\usepackage{epstopdf}
\usepackage{indentfirst}
\usepackage[title]{appendix}
\numberwithin{equation}{section}
\newtheorem{proposition}{Proposition}[section]
\newtheorem{lemma}[proposition]{Lemma}
\newtheorem{theorem}[proposition]{Theorem}

\newtheorem{rmk}[proposition]{Remark}

\renewcommand{\theequation}{\thesection.\arabic{equation}}
\renewcommand{\oddsidemargin}{5mm}

\def\figurename{Fig.} 
\makeatletter
\renewcommand{\fnum@figure}[1]{\figurename~\thefigure.}
\makeatother

\begin{document}
\parindent 15pt
\renewcommand{\theequation}{\thesection.\arabic{equation}}
\renewcommand{\baselinestretch}{1.15}
\renewcommand{\arraystretch}{1.1}
\renewcommand{\vec}[1]{\bm{#1}}
\title{Large-Amplitude Steady Electrohydrodynamic Solitary Waves with Constant Vorticity\thanks{T. Feng was supported by Postgraduate Research \& Practice Innovation Program of Jiangsu Province (No. KYCX24\_3914). Y. Zhang was supported by National Natural Science Foundation of China (No. 12301133), China Postdoctoral Science Foundation (No. 2023M741441, 2024T170353) and Jiangsu Education Department (No. 23KJB110007). Z. Zhang was supported by  National Key R\&D Program of China (No. 2022YFA1005601) and National Natural Science Foundation of China (No. 12031015).}
}
\author{Tingting Feng$^{a,}$\thanks{T.Feng,~~Email: 2112302204@stmail.ujs.edu.cn}~~~Yong Zhang$^{a,}$\thanks{Y.Zhang,~~Corresponding author. Email: 18842629891@163.com}~~~Zhitao Zhang$^{b,c,}$\thanks{Z.Zhang,~~Email: zzt@math.ac.cn}\\
	{\small $^a$ School of Mathematical Sciences, Jiangsu University, Zhenjiang 212013, P. R. China.}\\
{\small $^b$  HLM, Academy of Mathematics and Systems Science, Chinese Academy}\\
	{\small  of Sciences, Beijing 100190, P. R. China.}\\
	{\small $^c$  School of Mathematical Sciences,
		University of Chinese Academy of Sciences,}\\
	{\small Beijing 100049, P. R. China.}	
}
\date{}
\maketitle
\begin{abstract}
This paper investigates solitary water waves propagating along the surface of a two-dimensional dielectric fluid with constant vorticity in the presence of an external electric field. We formulate the system as a nonlinear free boundary problem where the Euler equations and electric potential equations are strongly coupled at the interface. A major challenge in such setting is the loss of standard monotonicity arguments due to the interaction between the velocity and electric fields. We overcome this difficulty by establishing new nodal properties for the combined system, ensuring the wave remains a symmetric elevation profile along the global branch. Moreover, along the global bifurcation curve, one of the following case must occur: (i) the formation of an equilibrium stagnation point, (ii) the degeneration of the conformal mapping, (iii) the onset of flow stagnation, or (iv) an unbounded increase in the dimensionless wave speed.
\end{abstract}

\section{Introduction}
Electrohydrodynamics (EHD) describes the interaction between electric fields and fluid flow fields \cite{DTP2019}. It has broad applications across chemistry, biology, and engineering. Recent developments include EHD conduction pumping for enhanced heat transfer in cooling systems \cite{EMGSG2006}; industrial coating processes, which are crucial in material production \cite{SFKPMS1997}; and electrospray ionization, a key technique for transforming solution-phase ions into highly charged gas-phase macromolecular ions \cite{JFDIGL1994}. Given the breadth of these applications, a deeper understanding of EHD behavior, particularly electrohydrodynamic interfacial waves, is of significant importance \cite{MVFTGRRAD2022, TGZWDP2024}.\par

Electrohydrodynamic interfacial waves driven by gravity and electric fields have been widely studied through weakly nonlinear theories \cite{GHPHDTP2007, HMJDD2021, Jiang, WZ2017} and numerical simulations \cite{DAGTV2020, DGV2022, GWV2023, LZW2017}. Earlier works \cite{TGZWDP2024, SVGVW2016} focused on wave singularities under normal electric fields, with or without surface tension. However, analytical results for fully nonlinear water waves with vorticity under electric fields remain limited. In a recent study \cite{DGXFZY2024}, a flattening method and local bifurcation theory were employed to establish the existence of small-amplitude periodic electrohydrodynamic waves, though the analysis was restricted to free surfaces in graph form. Numerical evidence \cite{SADVMH2019, VBJM1995, JASPGS1985, AFTDHP1988, JMVB1994} suggests the existence of overhanging profiles. In recent years, there has been notable progress in rigorously establishing the existence of overhanging profiles using perturbative techniques. For periodic waves, see \cite{VMHMHW2020, VMHMHW2022, FG2025}, and for solitary waves, see \cite{JDMDPMMMHW2024}. More recently, conformal mapping and Crandall-Rabinowitz local bifurcation theory were employed in \cite{DGTFZY2025} to establish the existence of periodic electrohydrodynamic waves, allowing for overhanging profiles.\par

 This paper presents the first construction of solitary electrohydrodynamic waves featuring constant vorticity and permitting overhanging profiles. Unlike in \cite{DGTFZY2025}, the linearized operator at our bifurcation point is not Fredholm, so the Crandall-Rabinowitz local bifurcation theorem used for periodic waves \cite{ACWS2004, Constantin, SW2009} does not apply. To overcome this, there are a great number of work on steady solitary waves has been done in \cite{CJAJFT1981, CJAJFT19812,MAL1954, JTB1977,RMCSWMHW2022, KOFDHH1954,Kozlov,RMCSWMHW2018,SVHMHW2023,MHW2013,MHW20152} without electric fields. Thus, we will use a center manifold reduction developed in \cite{RMCSWMHW2022,SVHMHW2023} to construct small-amplitude solutions.
 Moreover, the classical global analytic bifurcation results \cite{END1973, BBJT2003} require compactness, which does not hold here because the domain is unbounded. This issue is resolved by the global bifurcation framework in \cite[Theorem 6.1]{RMCSWMHW2018} or in \cite[Theorem B.1]{SVHMHW2023}, which is designed for solitary water waves. The inclusion of the electric field introduces substantial analytical complexities beyond those encountered in \cite{SVHMHW2023}. To this end, we construct a new working open set. This enables us to verify the required open condition and closed condition, thereby overcoming the fundamental difficulty pertaining to the nodal properties. Consequently, a new phenomenon maybe observed: {\bf an equilibrium stagnation point}, defined as a point in the fluid where all components of both the velocity and electric fields vanish. \par

On the other hand, we would like to mention some recent numerical works \cite{MVFTGRRAD2022, MVFTGRRAD2023, MVFEKRRNZ2024}, which explored the impact of electric fields on the flow beneath surface waves. In \cite{MVFTGRRAD2022}, it was shown that, under a normal electric field, stagnation points of periodic waves appear below the free surface. According to variations in the voltage potential, the flow may contain $0$, $2$, or $3$ stagnation points. For solitary waves, it was established by \cite{MVFTGRRAD2023} that the location of stagnation points is not significantly affected by variations in the electric field, while the electric field itself does not serve as a mechanism for their formation. Furthermore, increasing the strength of a horizontal electric field can eliminate stagnation points in periodic waves, as established by \cite{MVFEKRRNZ2024}.\par

In the following, we consider a two-dimensional, incompressible, inviscid flow in a system. The upper layer is conducting gas with a constant voltage potential, the lower layer is dielectric fluid with density $\rho = 1$ and permittivity $\epsilon_1 > 0$. Electrodes are placed at the top and bottom boundaries. Let $\Omega$ be the unbounded fluid domain in the $(X, Y)$-plane, bounded below by the planar electrode at $Y = 0$ and above by the free surface $\mathcal{S} = \{(\xi(s, d), \eta(s, d)) : s \in \mathbb{R} \}$, where $\xi$ and $\eta$ are defined in subsection \ref{sub2.1} and $\xi'(s)^2 + \eta'(s)^2 \neq 0$. As $|X| \to \infty$, we assume that $\Omega$ tends to a horizontal strip of depth $d$, called the asymptotic depth. See Figure \ref{025+figure1}.\par
\begin{figure}[htbp]
\centering
\includegraphics[width=9cm,height=4.8cm]{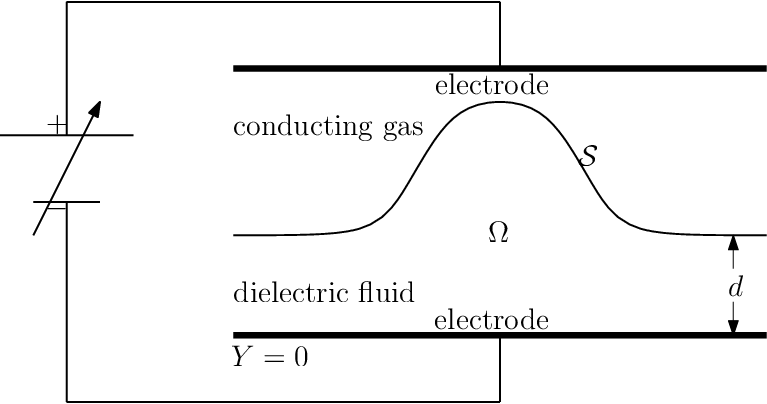}
\caption{Schematic of the problem.}
\label{025+figure1}
\end{figure}
We introduce a stream function $\psi(X, Y)$ such that the velocity field
\begin{align}
(U, V) = (\psi_Y, -\psi_X)\label{flow filed}
\end{align}
satisfies
\begin{eqnarray}
\left\{\begin{array}{llll}{\Delta\psi=\omega} & {\text { in }\Omega}, \\
{\psi=m} & {\text { on }\mathcal{S}}, \\
{\psi=0} & { \text { on }Y=0},\end{array}\right. \label{flow function equation}
\end{eqnarray}
where $\omega$ is the constant vorticity and $m=\int^{\eta(s)}_0\psi_Y(X,Y)dY$ is the relative mass flux.\par
In the electrostatic limit of Maxwell's equations, the induced magnetic field is negligible, making the electric field irrotational by Faraday's law. Thus, we introduce a potential function \(\varphi\) for the electric field satisfying
\begin{align}
\mathbf{E}=(E_1, E_2)=(\varphi_X, \varphi_Y).\label{electric filed}
\end{align}
In the dielectric fluid layer, \(\varphi\) satisfies Laplace equation. The potential difference between the electrodes is \(\varphi_0 > 0\), with boundary conditions \(\varphi = 0\) at the bottom electrode and \(\varphi = \varphi_0\) in the gas layer. That is,
\begin{eqnarray}
\left\{\begin{array}{llll}
{\Delta \varphi=0} & {\text { in }\Omega}, \\
{\varphi=\varphi_0} & {\text { on }\mathcal{S}}, \\
{\varphi=0} & { \text { on }Y=0}. \end{array}\right. \label{potential function equation}
\end{eqnarray}
The electrohydrodynamic system contains an additional energy term arising from the electric field. For a detailed formulation, see \cite[(2.7)]{TGZWDP2024} and \cite[(2.6)]{GWV2023}. The sign of this energy depends on both the region and the orientation of the electric field. Then, we have the following modified Bernoulli law
\begin{eqnarray}
P+\frac{1}{2}|\nabla\psi|^2+g(Y-d)+\frac{\epsilon_1}{2}|\nabla \varphi|^2-\omega\psi=\text{const} \quad\text{in }\Omega, \label{modified Bernoulli law}
\end{eqnarray}
where $P$ is pressure, \( g > 0 \) is the gravitational acceleration constant and $\epsilon_1$ is a dimensionless permittivity. Combining all the above considerations, we conclude that
\begin{eqnarray}
|\nabla\psi|^2+2g(Y-d)+\epsilon_1|\nabla \varphi|^2=(1+\epsilon_1)Q\quad\text{on }\mathcal{S}, \label{surface Bernoulli law}
\end{eqnarray}
where $Q$ is the Bernoulli constant. On the free surface, together with the asymptotic conditions
\begin{equation}
\begin{split}
&\psi_{X} \rightarrow 0, \quad \psi_{Y} \rightarrow F \sqrt{g d}(\gamma\frac{Y-d}{d}+1) \quad \text { as } X \rightarrow \pm \infty, \\
&\varphi_{X}\rightarrow 0, \quad \varphi _{Y} \rightarrow F \sqrt{g d}\quad \text { as } X \rightarrow \pm \infty,
 \end{split}\label{Asymptotic conditions for stream functions and potential functions}
\end{equation}
uniformly in $Y$. Here, $F$ is the Froude number, a dimensionless wave speed. From \eqref{Asymptotic conditions for stream functions and potential functions}, we see that $Q$, $m$, $\varphi_0$ and $F$ are related by
\begin{align}
Q = F^2gd,\quad m = Fg^{1/2}d^{3/2}(1-\frac{1}{2}\gamma),\quad\varphi_0=Fg^{1/2}d^{3/2}, \label{Q and F relationship}
\end{align}
the dimensionless measure $\gamma$ of $\omega$ is given by
\begin{align}
\gamma=\omega\frac{d^{1/2}}{Fg^{1/2}}. \notag
\end{align}
\par
The structure of the paper is organized as follows. In Section \ref{025section2}, we reformulate the free boundary problem into the forms suitable for further analysis and state the main results. Section \ref{025section3} formulates the problem as a nonlinear operator equation in an appropriate Banach space and examines the associated linearized operators. Section \ref{025section4} is devoted to establishing the existence theory for both small-amplitude solutions (see Theorem \ref{006theorem7.1}) and large-amplitude solutions (see Theorem \ref{006theorem7.4}). In Section \ref{025section5}, we present several qualitative results that facilitate the analysis of the global bifurcation structure and culminate in the proof of Theorem \ref{006theorem7.9}. 
\section{Reformulations}\label{025section2}
Let \(\Omega\) be a connected, open subset of \(\mathbb{R}^{n}\), possibly unbounded. We define \(\varphi \in C_{\mathrm{c}}^{\infty}(\overline{\Omega})\) as a function in \(C^{\infty}(\overline{\Omega})\) with compact support in \(\overline{\Omega}\). For \(k \in \mathbb{N}\) and \(\beta \in (0,1)\), we say \(u \in C^{k+\beta}(\overline{\Omega})\) if \(\|\varphi u\|_{C^{k+\beta}(\Omega)}<\infty\) for all \(\varphi\in C_{\mathrm{c}}^{\infty}(\overline{\Omega})\). We say that \(u_n\rightarrow u \) in \(C_{\mathrm{loc}}^{k+\beta}(\overline{\Omega})\) means \(\|\varphi(u_{n}-u)\|_{C^{k+\beta}(\Omega)}\to 0\) for all \(\varphi \in C_{\mathrm{c}}^{\infty}(\overline{\Omega})\). If \(u \in C^{k+\beta}(\overline{\Omega})\) and \(\|u\|_{C^{k+\beta}(\Omega)}<\infty\), then \(u\in C_{\mathrm{b}}^{k+\beta}(\overline{\Omega})\).\par
For unbounded \(\Omega\), we define \(C_{0}^{k}(\overline{\Omega}) \subseteq C_{\mathrm{b}}^{k}(\overline{\Omega})\) as
\[
C_{0}^{k}(\overline{\Omega}):=\{u \in C_{\mathrm{b}}^{k}(\overline{\Omega}): \lim _{r \to \infty} \sup _{|x|=r}|D^{j} u(x)|=0, \quad 0 \leq j \leq k\}.
\]
We also define weighted H\"{o}lder spaces allowing exponential growth in the \(x_{1}\)-direction. For \(\mu>0\), \(u\in C_{\mu}^{k+\beta}(\overline{\Omega})\) if \(\|u\|_{C_{\mu}^{k+\beta}(\Omega)}<\infty\), where
\[
\|u\|_{C_{\mu}^{k+\beta}(\Omega)}:=\sum_{|\alpha| \leq k}\|\operatorname{sech}(\mu x_{1}) \partial^{\alpha} u\|_{C^{0}(\Omega)}+\sum_{|\alpha|=k}\|\operatorname{sech}(\mu x_{1})|\partial^{\alpha} u|_{\beta}\|_{C^{0}(\Omega)}
\]
with the local H\"{o}lder seminorm
\[
|u|_\beta(x):=\sup _{|y|<1} \frac{|u(x+y)-u(x)|}{|y|^{\beta}}.
\]
\subsection{Conformal mapping}\label{sub2.1}
Introducing the conformal mapping \(X + iY = \xi(x, y) + i\eta(x, y)\), the physical fluid domain \(\Omega\) is transformed into the rectangular strip
\[
\mathcal{R} = \{(x, y) \in \mathbb{R}^2 : 0 < y < d\}.
\]
We denote the upper and lower boundaries of $\mathcal{R}$ by
\[
\Gamma = \{(x, y) \in \mathbb{R}^2 : y = d\}, \quad \mathcal{B} = \{(x, y) \in \mathbb{R}^2 : y = 0\},
\]
see Figure \ref{025figure2}.
\begin{figure}[htbp]
\centering
\includegraphics[width=14cm,height=5cm]{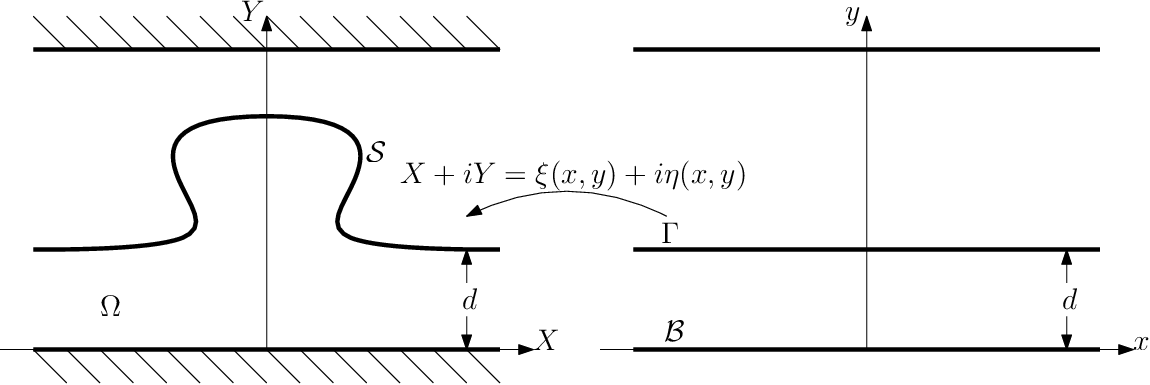}
\caption{The conformal parametrization of the fluid domain $\Omega$}
\label{025figure2}
\end{figure}
The mapping is normalized by the condition
\[
\xi_x + i \eta_x \to 1 \quad \text{as } x \to \infty.
\]
\par
We reformulate \eqref{flow filed}-\eqref{Asymptotic conditions for stream functions and potential functions} in conformal variables. For convenience, we define $\Psi, \theta: \mathcal{R}\rightarrow \mathbb{R}$ by
\begin{align}
\Psi(x, y) = \psi(\xi(x, y), \eta(x, y)), \quad
\theta(x, y) = \varphi(\xi(x, y), \eta(x, y))
\label{Conformal variable representation}
\end{align}
for all \((x, y) \in \mathcal{R}\), and introduce $\zeta, \vartheta: \mathcal{R} \rightarrow \mathbb{R}$ by
\begin{align}
\zeta(x, y) = \Psi(x, y) - \frac{1}{2} \omega \eta^2(x, y),\quad\vartheta(x, y) = \theta(x, y).\label{New variable representation}
\end{align}
Then, equations \eqref{flow filed}-\eqref{Asymptotic conditions for stream functions and potential functions} would be transformed into the system
\begin{subequations}\label{Conformal mapping equations 1}
\begin{align}
\Delta \zeta = 0  &\quad\text{ in } \mathcal{R},\label{025 Conformal mapping equation a}\\
\Delta \vartheta = 0  &\quad\text{ in } \mathcal{R},\label{025 Conformal mapping equation b}\\
\zeta = m - \frac{1}{2} \omega \eta^2 &\quad\text{ on } \Gamma,\label{025 Conformal mapping equation c}\\
(\zeta_y + \omega \eta \eta_y)^2 + \epsilon_1 \vartheta_y^2 = \big((1+\epsilon_1)Q- 2g(\eta - d)\big) |\nabla \eta|^2 &\quad\text{ on } \Gamma,\label{025 Conformal mapping equation d}\\
\vartheta = \varphi_0 &\quad\text{ on } \Gamma,\label{025 Conformal mapping equation e}\\
\zeta = 0 &\quad\text{ on } \mathcal{B},\label{025 Conformal mapping equation f}\\
\vartheta = 0 &\quad\text{ on } \mathcal{B}\label{025 Conformal mapping equation g},
\end{align}
coupled with the conditions
\begin{align}
\Delta \eta = 0 &\quad\text{ in } \mathcal{R},\label{Conformal mapping equation 2}\\
\eta = 0 &\quad\text{ on } \mathcal{B}\label{Conformal mapping equation 3}.
\end{align}
The asymptotic conditions become
\begin{align}
\lim_{x \to \pm \infty} \zeta(x, y) = (m - \frac{1}{2} \omega d^2) \frac{y}{d}, \quad
\lim_{x \to \pm \infty} \eta(x, y) = y, \quad
\lim_{x \to \pm \infty} \vartheta(x, y) = \varphi_0\frac{y}{d},\label{Conformal mapping equations 4}
\end{align}
uniformly in \( y \).
\end{subequations}
To derive \eqref{025 Conformal mapping equation d} on the left-hand side, we use
\begin{align}
U\eta_x-V\xi_x=0,\quad E_1\xi_x+E_2\eta_x=0,\notag
\end{align}
where $U, V, E_1, E_2$ are given in \eqref{flow filed} and \eqref{electric filed}, see also \cite{SVHMHW2023, HMJ2013} for related formulations.
\subsection{Non-dimensionalization}
In this subsection, we use rescaling technique to reduce (\ref{Conformal mapping equations 1}) to a single-parameter problem, where the parameter $\alpha$ and its critical value $\alpha_{\mathrm{cr}}$ are defined as
\allowdisplaybreaks
\begin{align}
\alpha=\frac{1}{F^{2}}, \quad \alpha_{\mathrm{cr}}=\frac{1}{F_{\mathrm{cr}}^{2}}=1-\gamma +\epsilon_1,\notag
\end{align}
with $\gamma$ being a dimensionless measure of the vorticity $\omega$. A detailed derivation of $\alpha_{\mathrm{cr}}$ is provided in subsection \ref{025subsection5.4}.\par
 From \eqref{Q and F relationship}, we take $d$ as the length scale and $F \sqrt{g d}$ as the velocity scale. For simplicity, we use dimensionless variables $\bar{\eta}$, $\bar{\zeta}$, and $\bar{\vartheta}$, omitting the tildes. With this scaling, \eqref{025 Conformal mapping equation a}-\eqref{Conformal mapping equation 3} would be transformed into their dimensionless form
\begin{subequations}\label{yuNon-dimensionalization equations 1}
\begin{align}
\Delta\eta=0&\quad\text{ in } \mathcal{R},\label{yuNon-dimensionalization equation a}\\
\Delta\zeta=0&\quad\text{ in } \mathcal{R},\label{yuNon-dimensionalization equation b}\\
\Delta\vartheta=0&\quad\text{ in } \mathcal{R},\label{yuNon-dimensionalization equation c}\\
\zeta=1-\frac{1}{2}\gamma-\frac{1}{2} \gamma\eta^{2}&\quad\text{ on } \Gamma,\label{yuNon-dimensionalization equation d}\\
{(\zeta_y+\gamma\eta\eta_y)}^{2}+\epsilon_1\vartheta_y^2=\big(1+\epsilon_1-2\alpha(\eta-1)\big)|\nabla\eta|^{2}&\quad\text{ on } \Gamma,\label{yuNon-dimensionalization equation e}\\
\vartheta=1&\quad\text{ on } \Gamma,\label{yuNon-dimensionalization equation f}\\
\eta=0 &\quad\text{ on } \mathcal{B},\label{yuNon-dimensionalization equation g}\\
\zeta=0&\quad\text{ on } \mathcal{B},\label{yuNon-dimensionalization equation h}\\
\vartheta=0&\quad\text{ on } \mathcal{B}.\label{yuNon-dimensionalization equation i}
\end{align}
\end{subequations}
From \eqref{yuNon-dimensionalization equation c}, \eqref{yuNon-dimensionalization equation f}, and \eqref{yuNon-dimensionalization equation i}, it follows that $\vartheta\equiv y$, a relation that directly impacts the structure of \eqref{yuNon-dimensionalization equation e}. Consequently, we rewrite \eqref{yuNon-dimensionalization equations 1} as
\begin{subequations}\label{Non-dimensionalization equations 1}
\begin{align}
\Delta\eta=0&\quad\text{ in } \mathcal{R},\label{Non-dimensionalization equation a}\\
\Delta\zeta=0&\quad\text{ in } \mathcal{R},\label{Non-dimensionalization equation b}\\
\zeta=1-\frac{1}{2}\gamma-\frac{1}{2} \gamma\eta^{2}&\quad\text{ on } \Gamma,\label{Non-dimensionalization equation d}\\
{(\zeta_y+\gamma\eta\eta_y)}^{2}+\epsilon_1=\big(1+\epsilon_1-2\alpha(\eta-1)\big)|\nabla\eta|^{2}&\quad\text{ on } \Gamma,\label{Non-dimensionalization equation e}\\
\eta=0 &\quad\text{ on } \mathcal{B},\label{Non-dimensionalization equation g}\\
\zeta=0&\quad\text{ on } \mathcal{B},\label{Non-dimensionalization equation h}\\
\vartheta&\equiv y,\label{Non-dimensionalization equation f add}
\end{align}
we additionally assume that the regularity condition
\begin{align}
\eta, \zeta \in C_{\mathrm{b}}^{3+\beta}(\overline{\mathcal{R}}),\label{Non-dimensionalization equation 2}
\end{align}
and the symmetry condition
\begin{align}
\eta\text { and }\zeta \text { are even in } x.\label{Non-dimensionalization equation 3}
\end{align}
The asymptotic conditions \eqref{Conformal mapping equations 4} turn into
\[
\eta = y, \quad \zeta = (1 - \gamma) y.
\]\par
We define
\begin{align}
w= (w_1, w_2), \quad w_1 = \eta - y, \quad w_2 = \zeta - (1 - \gamma)y,
\label{w equations 1}
\end{align}
 and require
\begin{align}
w \in C_0^2(\overline{\mathcal{R}})\label{w equation 2}.
\end{align}
Moreover, we suppose that
\begin{align}
\inf_{\Gamma} (1 +\epsilon_1- 2\alpha (\eta - 1)) |\nabla \eta|^2>\epsilon_1,\quad \inf_{\mathcal{R}} (1 +\epsilon_1- 2\alpha (\eta - 1))^2 |\nabla \eta|^2> 0.\label{Regional boundary condition}
\end{align}
\end{subequations}
The nonvanishing of $(1 +\epsilon_1- 2\alpha (\eta - 1)) |\nabla \eta|^2-\epsilon_1$ prevents the wave from reaching its greatest possible height, and $\nabla \eta$ being nonzero ensures the conformal mapping remains non-degenerate---specifically, that the free surface does not shrink to a point.
The first inequality provides a sufficient condition for proving the nodal properties. The second inequality provides a basis for the later application of linear Schauder estimates, where the first factor not vanishing implies that we cannot have an equilibrium stagnation.

\subsection{Formulation of the velocity field and electric field}
The velocity components $\psi_Y = U$ and $\psi_X = -V$ are expressed in terms of conformal variables $x$ and $y$ by
$$
u(x, y) := U(\xi(x, y), \eta(x, y)) \quad \text{and} \quad v(x, y) := V(\xi(x, y), \eta(x, y)).
$$
Similarly, the electric field components $\varphi_X =E_1$ and $\varphi_Y = E_2$ are represented as
$$
e_1(x, y) := E_1(\xi(x, y), \eta(x, y)) \quad \text{and} \quad {e_2}(x, y) := E_2(\xi(x, y), \eta(x, y)).
$$
Using the chain rule along with \eqref{Conformal variable representation} and \eqref{New variable representation}, we derive from \eqref{Non-dimensionalization equation f add} that
\begin{equation}
\begin{split}
(u, v) &= \left( \frac{\eta_x \zeta_x + \eta_y \zeta_y}{\eta_x^2 + \eta_y^2} + \gamma \eta,\; \frac{\eta_x \zeta_y - \eta_y \zeta_x}{\eta_x^2 + \eta_y^2} \right), \\
({e_1}, {e_2}) &=\left(\frac{\eta_y \vartheta_x-\eta_x \vartheta_y}{\eta_x^2 + \eta_y^2},\; \frac{\eta_x \vartheta_x + \eta_y \vartheta_y}{\eta_x^2 + \eta_y^2}\right)=\left(\frac{-\eta_x}{\eta_x^2 + \eta_y^2},\; \frac{\eta_y}{\eta_x^2 + \eta_y^2}\right).
\end{split} \label{Velocity field and electric field}
\end{equation}
The condition \eqref{Regional boundary condition} ensures that the denominator in \eqref{Velocity field and electric field} is well-defined. Moreover, both $\zeta_y + i\zeta_x$ and $\eta_y + i\eta_x$ are holomorphic functions. Then
\begin{align}
\frac{\zeta_y+i\zeta_x}{\eta_y+i\eta_x}=
\frac{\eta_{x} \zeta_{x}+\eta_{y} \zeta_{y}}{\eta_{x}^{2}+\eta_{y}^{2}}
+i\frac{\eta_{y} \zeta_{x}-\eta_{x} \zeta_{y}}{\eta_{x}^{2}+\eta_{y}^{2}}\notag
\end{align}
is a holomorphic function, its real and imaginary parts are harmonic. Therefore, the velocity field components $u$ and $v$ are harmonic functions in $\mathcal{R}$. Similarly,
\begin{align}
\frac{1}{\eta_y+i\eta_x}=
\frac{\eta_{y}}{\eta_{x}^{2}+\eta_{y}^{2}}
+i\frac{-\eta_{x}}{\eta_{x}^{2}+\eta_{y}^{2}}\notag
\end{align}
is a holomorphic function. Then, the electric field components $e_1$ and $e_2$ are harmonic functions in $\mathcal{R}$.\par
It is straightforward to verify that the functions \(u, v\) and \({e_1}, {e_2}\) satisfy
\begin{subequations}\label{Velocity field and electric field equations 1}
\begin{align}
u_x+v_y=\gamma\eta_x &\quad\text{ in } \mathcal{R},\label{Velocity field and electric field equation a}\\
{e_1}_y-{e_2}_x=0 &\quad\text{ in } \mathcal{R},\label{Velocity field and electric field equation b}\\
u_y-v_x=\gamma\eta_y &\quad\text{ in } \mathcal{R},\label{Velocity field and electric field equation c}\\
{e_1}_x+{e_2}_y=0 &\quad\text{ in } \mathcal{R},\label{Velocity field and electric field equation d}\\
u\eta_x-v\eta_y=0 &\quad\text{ on } \Gamma,\label{Velocity field and electric field equation e}\\
{e_1}\eta_y+{e_2}\eta_x=0 &\quad\text{ on } \Gamma,\label{Velocity field and electric field equation f}\\
u^2+v^2+\epsilon_1({e_1}^2+{e_2}^2)+2\alpha (\eta-1)=1+\epsilon_1&\quad\text{ on }\Gamma,\label{Velocity field and electric field equation g}\\
v=0&\quad\text{ on }\mathcal{B},\label{Velocity field and electric field equation h}\\
{e_1}=0&\quad\text{ on }\mathcal{B}.\label{Velocity field and electric field equation i}
\end{align}
The asymptotic behavior of the velocity field and the electric field are given by
\begin{align}
\lim_{x \to \pm \infty} u = (1 - \gamma) + \gamma y, \quad \lim_{x \to \pm \infty} v = 0, \quad \lim_{x \to \pm \infty} {e_1} = 0, \quad \lim_{x \to \pm \infty} {e_2} = 1. \label{Velocity field and electric field equations 2}
\end{align}
From the regularity condition \eqref{Non-dimensionalization equation 2}, it follows that
\begin{align}
u, v, {e_1}, {e_2} \in C_{\mathrm{b}}^{2+\beta}(\overline{\mathcal{R}}), \label{Velocity field and electric field equations 3}
\end{align}
and from \eqref{Non-dimensionalization equation 3}, we have
\begin{align}
u, e_2 \text{ are even in } x, \quad v, {e_1}\text{ are odd in } x. \label{Velocity field and electric field equations 4}
\end{align}
\end{subequations}
\begin{rmk}\label{025remk2.1}
From \eqref{Regional boundary condition} and \eqref{Velocity field and electric field}, it follows that
$$ 1+\epsilon_1-2\alpha(\eta-1)>\frac{\epsilon_1}{|\nabla\eta|^2}=\epsilon_1(e_1^2+e_2^2).
$$
Consequently, by \eqref{Velocity field and electric field equation g}, we conclude that $u$ and $v$ do not vanish simultaneously at any point on $\Gamma$. This non-vanishing property of the velocity field plays a crucial role in the derivation of the nodal properties.
\end{rmk}

\subsection{Statement of main results}\label{025section2.4}
The main theorem of the paper is the following existence result for large-amplitude solitary electrohydrodynamic waves, which allows for internal stagnation points and overhanging profiles.
\begin{theorem} \label{006theorem7.9ABOVE}
Fix the gravitational constant \(g > 0\), the asymptotic depth \(d > 0\), $\gamma<0$ and permittivity $\epsilon_1 > 0$. Then there exists a global continuous curve \(\mathscr{C}\) of solutions to \eqref{flow filed}-\eqref{Asymptotic conditions for stream functions and potential functions}, parameterized by \(s \in (0, \infty)\). Moreover, one of the following asymptotic property holds along \(\mathscr{C}\) as \(s \to \infty\):

(i) (Equilibrium stagnation) $\inf_{\Gamma} \bigl(1+\epsilon_1 - \frac{2}{F^2} \frac{\eta - d}{d} \bigr)\longrightarrow 0$; or

(ii) $\inf_{\Gamma}|\nabla\eta(s)|\rightarrow 0$, the conformal transformation of variables degenerates and the free surface may locally shrink to a point, leading to a singularity; or

(iii) (Stagnation) $\inf_\Gamma \biggl(\bigl(1 + \epsilon_1 - \frac{2}{F^2} \frac{\eta - d}{d}\bigr) |\nabla \eta(s)|^2-\epsilon_1\biggr)\longrightarrow 0$; or

(iv) $\frac{1}{F(s)}\rightarrow 0$, the dimensionless wave speed tends toward infinity.
\end{theorem}

\begin{rmk}	
Let us briefly explain the alternative (i) and (iii) in Theorem \ref{006theorem7.9ABOVE}. If the Froude number $F$ remains bounded,

(i)
if $\inf_{\Gamma} \big( 1 + \epsilon_1 - \frac{2}{F^2} \frac{\eta - d}{d} \big) \to 0$,
it follows from \eqref{Velocity field and electric field equation g} that
$$
\inf_{\Gamma} \left(1 + \epsilon_1 - \frac{2}{F^2} \frac{\eta - d}{d}\right) =\inf_{\Gamma} \left( \frac{U^2+V^2 + \epsilon_1({E_1}^2+{E_2}^2)}{F^2 g d}\right) \rightarrow 0.
$$
This condition implies that $U \to 0, V \to 0$ and $E_1 \to 0, E_2 \to 0$, that is to say, an equilibrium stagnation point occurs.

(iii)
if $\inf_\Gamma \biggl(\bigl(1 + \epsilon_1 - \frac{2}{F^2} \frac{\eta - d}{d}\bigr) |\nabla \eta(s)|^2-\epsilon_1\biggr)\to 0$, it follows from \eqref{Non-dimensionalization equation e} that
$$\inf_\Gamma \biggl(\bigl(1 + \epsilon_1 - \frac{2}{F^2} \frac{\eta - d}{d}\bigr) |\nabla \eta(s)|^2-\epsilon_1\biggr)=\inf_\Gamma\bigl(\frac{U^2}{F^2 g d}\bigr)\longrightarrow 0.$$
This condition implies that the fluid approaches stagnation $(U=V=0)$ at the crest.
\end{rmk}
We also derive several qualitative properties of solitary electrohydrodynamic waves. These results are not only essential for proving Theorem \ref{006theorem7.9ABOVE}, but are also of independent mathematical interest. Below, we emphasize the two most important findings, additional properties are discussed in Section \ref{025section5}.\par
The first key result demonstrates the nonexistence of monotone bores.
An electrohydrodynamic bore is a type of traveling wave that asymptotically tends different laminar flows as \( x \to \pm\infty \), see Figure \ref{025figure3}. In previous studies on solitary waves, bores have often been considered as an alternative limiting behavior to extreme waves. This behavior is commonly associated with a loss of compactness. In this work, we aim to rule out the existence of such bores.
\par
\begin{figure}[htbp]
\centering
\includegraphics[width=7.6cm,height=4.4cm]{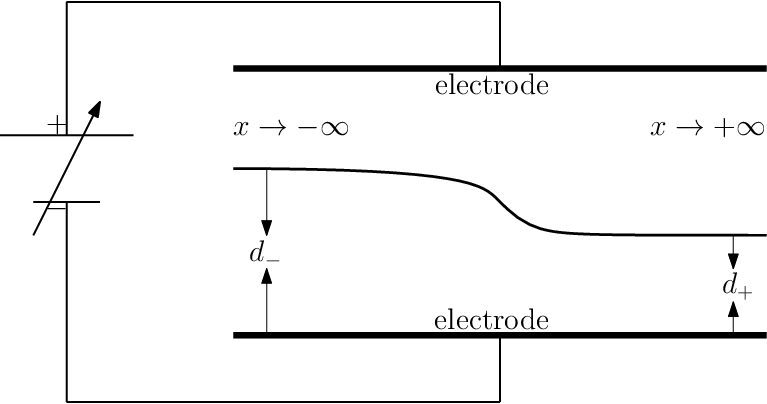}
\caption{An electrohydrodynamic bore, with distinct asymptotic velocities, electric fields and depths at $x \to \pm\infty$.}
\label{025figure3}
\end{figure}
An electrohydrodynamic bore is defined as a solution \((\eta, \zeta, \vartheta, \alpha)\) to equations \eqref{Non-dimensionalization equation a}-\eqref{Non-dimensionalization equation 2} that satisfies the asymptotic condition
\begin{align}
\lim_{x \to \pm \infty} (\eta, \zeta, \vartheta)(x, y) = (\eta_\pm(y), \zeta_\pm(y), \vartheta_\pm(y)), \label{bores asymptotic condition}
\end{align}
where the limiting profiles \((\eta_-, \zeta_-, \vartheta_-)\) and \((\eta_+, \zeta_+, \vartheta_+)\) are distinct. By a translation argument, we observe that these limits must also solve \eqref{Non-dimensionalization equation a}-\eqref{Non-dimensionalization equation 2}. Therefore, they take the form
\[
\eta_{\pm}(y) = \hat{\eta}_{\mathrm{tr}}(y; d_{\pm}), \quad \zeta_{\pm}(y) = \hat{\zeta}_{\mathrm{tr}}(y; d_{\pm}), \quad \vartheta_{\pm}(y) = \hat{\vartheta}_{\mathrm{tr}}(y; d_{\pm}),
\]
for constants \( d_- \neq d_+ \), where
\begin{equation}
\begin{split}
\hat{\eta}_{\mathrm{tr}}(y; d) & := dy, \\
\hat{\zeta}_{\mathrm{tr}}(y; d) & := \big(\frac{2 - \gamma}{2d} - \frac{\gamma d}{2}\big) dy, \\
\hat{\vartheta}_{\mathrm{tr}}(y; d) & :=\frac{1}{d}dy=y.
\end{split} \label{0063.12}
\end{equation}
It is crucial to note that the solutions of \eqref{Non-dimensionalization equations 1} satisfy \eqref{bores asymptotic condition}-\eqref{0063.12} with \( d_+ = d_- = 1 \).
\par
Moreover, since \eqref{0063.12} satisfies the dynamic boundary condition \eqref{Non-dimensionalization equation e}, it follows that
\begin{align}
\hat{Q}(d_{-}) = \hat{Q}(d_{+}) = \hat{Q}(1)=1+\epsilon_1, \label{0063.13}
\end{align}
where the function \(\hat{Q}(d)\) is defined by
\begin{align}
\hat{Q}(d) := \frac{1}{d^2}\left( \frac{2 - \gamma}{2} + \frac{\gamma d^2}{2}\right)^2 +\frac{1}{d^2}\epsilon_1+ 2\alpha (d - 1). \label{0063.14}
\end{align}
We introduce the function
\begin{align}
\hat{S}(d) := S\big(x; (\hat{\eta}_{\mathrm{tr}}(\cdot; d), \hat{\zeta}_{\mathrm{tr}}(\cdot; d), \hat{\vartheta}_{\mathrm{tr}}(\cdot; d))\big), \label{0063.15}
\end{align}
where \( S \) denotes the flow force as given in \eqref{0063.4}. The invariance of the flow force implies
\begin{align}
\hat{S}(d_{-}) = \hat{S}(d_{+}) = \hat{S}(1). \label{0063.16}
\end{align}
Equations \eqref{0063.13} and \eqref{0063.16} are known as the conjugate flow conditions.
\begin{theorem}[Nonexistence of bore solutions]\label{006theorem3.4}
The conjugate flow equations \eqref{0063.13} and \eqref{0063.16} have no solutions other than \(d = 1\). In particular, the system defined by \eqref{Non-dimensionalization equation a}-\eqref{Non-dimensionalization equation 2} does not admit bore solutions as defined in \eqref{bores asymptotic condition}-\eqref{0063.12}.
\end{theorem}
The second main result provides lower bound estimates for the Froude number.
\begin{theorem}\label{006theorem3.8}
Let \((\eta, \zeta, \vartheta, \alpha)\) be a solution to \eqref{Non-dimensionalization equations 1} with \(\alpha > 0\). If \(\eta \geq 1\) on the surface \(\Gamma\), then either
\begin{align}
\alpha = \frac{1}{F^{2}} < 1 - \gamma + \epsilon_1 \notag
\end{align}
holds, or the solution is trivial.
\end{theorem}
\begin{rmk}
To the best of our knowledge, this is the first bound of its kind in electrohydrodynamic waves.
\end{rmk}
\section{Linearized operators}\label{025section3}
We rewrite \eqref{Non-dimensionalization equations 1} as a nonlinear operator equation for $w = (w_1, w_2)$ in an appropriate Banach space. This section focuses on the analysis of the linearized operator $\mathscr{F}_{w}(w, \alpha)$.
\subsection{Functional analytic formulation}\label{025section3.1}
Define the Banach spaces
\begin{align}
&\mathcal{X} = \{ w \in (C_{\mathrm{b}}^{3+\beta}(\overline{\mathcal{R}})\cap C_{0}^{2}(\overline{\mathcal{R}}))^2: \Delta w = 0 \, \text{in} \, \mathcal{R}, \, w = 0 \, \text{on} \, \mathcal{B} \},\notag\\
&\mathcal{Y} = (C_{\mathrm{b}}^{3+\beta}(\Gamma) \cap C_{0}^{2}(\Gamma)) \times (C_{\mathrm{b}}^{2+\beta}(\Gamma) \cap C_{0}^{1}(\Gamma))\notag
\end{align}
and let \(\mathcal{X} \subset \mathcal{X}_{\mathrm{b}}\) and \(\mathcal{Y} \subset \mathcal{Y}_{\mathrm{b}}\), where  $\mathcal{X}_{\mathrm{b}}$ and $\mathcal{Y}_{\mathrm{b}}$ are larger spaces whose elements do not vanish at infinity.
That is, \(\mathcal{X}_{\mathrm{b}}\) and \(\mathcal{Y}_{\mathrm{b}}\) no longer intersect with \( C_0^1 \) and \( C_0^2 \).\par
Equation \eqref{Non-dimensionalization equations 1} can be reformulated as a nonlinear operator equation
\begin{align}
\mathscr{F}(w, \alpha) = 0,\label{0253.1}
\end{align}
where the mapping \(\mathscr{F}:= (\mathscr{F}_{1}, \mathscr{F}_{2}): \mathcal{X} \times \mathbb{R} \to \mathcal{Y}\) is defined by
\begin{align}
\mathscr{F}_{1}(w, \alpha) &= (w_{2} + \gamma w_{1} + \frac{1}{2} \gamma w_{1}^{2}), \notag\\
\mathscr{F}_{2}(w, \alpha) &= \big( \gamma(w_{1} + w_{1y} + w_{1} w_{1y}) +w_{2y}+1\big)^{2}+\epsilon_1 \notag\\
&\quad\quad- (1+\epsilon_1 - 2\alpha w_{1}) \big( w_{1x}^{2} + (w_{1y} + 1)^{2} \big). \notag
\end{align}
Considering the assumption \eqref{Regional boundary condition}, we seek solutions in the open subset
\begin{align}
\mathcal{U}: = \left\{ (w, \alpha) \in \mathcal{X} \times \mathbb{R} : \alpha < 1-\gamma+\epsilon_1,\, \kappa(w, \alpha)>\epsilon_1, \, \lambda(w, \alpha) > 0 \right\} \subset \mathcal{X} \times \mathbb{R},\notag
\end{align}
where
\begin{equation}
	\begin{split}
	\kappa(w, \alpha)&:= \inf_{\Gamma} (1 +\epsilon_1- 2\alpha w_1) \big(w_{1x}^2 + (1 + w_{1y})^2\big)>\epsilon_1,\\
\lambda(w, \alpha)&:= \inf_{\mathcal{R}} 4(1 +\epsilon_1 - 2\alpha w_1)^2 \big(w_{1x}^2 + (1 + w_{1y})^2\big)> 0.
\end{split}\label{0253.2}
\end{equation}
Notably, the first inequality in \eqref{0253.2} provides a sufficient condition to prove the nodal properties, the second allows for the application of linear Schauder estimates to the linearized problem.\par
By linearizing the operator, we obtain
\begin{equation}
\begin{split}
&\mathscr{F}_{1w}(w, \alpha) \dot{w} = c_{i} \dot{w}_{i},\\
&\mathscr{F}_{2w}(w, \alpha) \dot{w} = a_{ij} \partial_{j} \dot{w}_{i} + b_{i} \dot{w}_{i},
\end{split}\label{0253.3}
\end{equation}
where \(\partial_{1} = \partial_{x}\) and \(\partial_{2} = \partial_{y}\), and the coefficients are given by
\begin{align}
&b_{1}=2\gamma (w_{1y} + 1) \big( w_{2y} + 1+ \gamma (w_{1} + w_{1y} + w_{1} w_{1y}) \big) + 2\alpha (w_{1x}^{2} + (1 + w_{1y})^{2}), \notag\\
&b_{2}=0, \notag\\
&a_{11}=-2(1+\epsilon_1  - 2\alpha  w_{1}) w_{1x}, \notag\\
&a_{21}=0, \notag\\
&a_{12}=2 \gamma\big( w_{2y} + 1 + \gamma (w_{1} + w_{1y} + w_{1} w_{1y}) \big)(1 + w_{1})-2(1 +\epsilon_1 - 2\alpha w_{1})(1 + w_{1y}),\notag\\
&a_{22}=2\big(w_{2y} + 1+ \gamma (w_{1} + w_{1y} + w_{1} w_{1y})\big), \notag\\
&c_{1}=\gamma (1 + w_{1}), \notag\\
&c_{2}=1.\notag
\end{align}
We note that \eqref{0253.3} is the type of operator considered in appendix. Linearizing about the trivial solution \(w \equiv 0\), we have
\begin{subequations}\label{0253.4}
\begin{align}
&\mathscr{F}_{1w}(0, \alpha) \dot{w} = \gamma \dot{w}_{1} + \dot{w}_{2}, \label{0253.4a}\\
&\mathscr{F}_{2w}(0, \alpha) \dot{w} = 2\big(\dot{w}_{2y} + (\gamma -1-\epsilon_1)\dot{w}_{1y} + (\gamma +\alpha) \dot{w}_{1}\big).\label{0253.4b}
\end{align}
\end{subequations}
\begin{rmk}\label{rmk1}
By \eqref{w equation 2}, the limit operator of \(\mathscr{F}_{w}(w, \alpha)\) as \(x \to \pm \infty\) is \(\mathscr{F}_{w}(0, \alpha)\).
\end{rmk}
\subsection{Local properness and invertibility properties}
We begin by showing that the linearized operator $\mathscr{F}_w(w, \alpha)$ satisfies the conditions necessary for applying Schauder estimates.
\begin{lemma}\label{006lemma5.3} For \((w, \alpha) \in \mathcal{U}\), the linearized operator \(\mathscr{F}_w(w, \alpha): \mathcal{X}_{\mathrm{b}} \to \mathcal{Y}_{\mathrm{b}}\) in \eqref{0253.3} satisfies the Schauder estimate
\[
\|\dot{w}\|_{\mathcal{X}_{\mathrm{b}}} \leq C \big( \|\mathscr{F}_{w}(w, \alpha) \dot{w}\|_{\mathcal{Y}_{\mathrm{b}}} + \|\dot{w}\|_{C^{0}(\mathcal{R})} \big),
\]
where the constant \(C\) depends on \(\|w\|_{\mathcal{X}}\), \(\alpha\) and the minor constant \(\lambda(w, \alpha)\). A more intuitive explanation of \(\lambda(w, \alpha)\) is given in \cite[Theorem A.1]{SVHMHW2023}.
\end{lemma}

\begin{proof}
To apply \cite[Theorem A.1]{SVHMHW2023}, we must verify that
\[
\begin{split}
\inf_{\Gamma} &\Big( (c_1a_{21}-c_2a_{11})^2 + (c_1a_{22}-c_2a_{12})^2 \Big)> 0.
\end{split}
\]
Based on coefficients given in subsection \ref{025section3.1} and condition \eqref{0253.2}, we derive
\[
\begin{split}
&\inf_{\Gamma} \Big( (c_1a_{21}-c_2a_{11})^2 + (c_1a_{22}-c_2a_{12})^2 \Big) \\
& = \inf_\Gamma 4(1 +\epsilon_1 - 2\alpha w_1)^2 \big(w_{1x}^2 + (1 + w_{1y})^2 \big)> 0.
\end{split}
\]
This completes the proof.
\end{proof}
For solitary electrohydrodynamic waves, the linearized operator \(\mathscr{F}_w(w, \alpha_{\mathrm{cr}})\) is not a Fredholm map from \(\mathcal{X} \rightarrow \mathcal{Y}\). In particular, its range is not closed (see \cite{VV2011}). To overcome this difficulty, we employ a center manifold reduction method. However, before proceeding, we first need to establish several preparatory results.
\begin{lemma}[Injectivity]\label{006lemma5.4}
The linear operator \(\mathscr{F}_w(0, \alpha): \mathcal{X}_{\mathrm{b}} \to \mathcal{Y}_{\mathrm{b}}\) defined in \eqref{0253.4} has a trivial kernel if and only if \(\alpha < 1 - \gamma+\epsilon_1 \).
\end{lemma}
\begin{proof}
Using separation of variables, it suffices to rule out solutions \((\dot{w}_1, \dot{w}_2) \in \mathcal{X}_{\mathrm{b}}\) of the form
\begin{align}
\dot{w}_{1} = c_{1} \cos(kx) \sinh(ky), \quad \dot{w}_{2} = c_{2} \cos(kx) \sinh(ky),\label{0065.8}
\end{align}
where \( k \) is a real wave number and \( c_1, c_2 \in \mathbb{R}\). Substituting \eqref{0065.8} into \eqref{0253.4a}, we obtain \( c_2 = -c_1 \gamma \). Inserting \eqref{0065.8} into \eqref{0253.4b}, we find that nontrivial solutions exist if and only if the dispersion relation
\begin{align}
\gamma + \alpha =(1+\epsilon_1 ) k \coth(k) \label{0065.9}
\end{align}
is satisfied. Observe that \( k \coth(k) \) attains its minimum value of 1 at \( k = 0 \). Therefore, \eqref{0065.9} admits no real solution whenever \( \alpha < 1 - \gamma+\epsilon_1\).
\end{proof}
\begin{rmk}
(i) Lemma \ref{006lemma5.4} demonstrates that the linearized operator \(\mathscr{F}_w(0, \alpha)\) is injective when mapping from \(\mathcal{X}_{\mathrm{b}}\) to \(\mathcal{Y}_{\mathrm{b}}\), given that \(\alpha < 1 - \gamma+\epsilon_1\). \par
(ii) If \(\alpha = 1 - \gamma+\epsilon_1\), the linear operator \(\mathscr{F}_w(0, \alpha)\) becomes singular.
\end{rmk}
We proceed to analyze the operator's local properness and its invertibility properties.
\begin{lemma}[Local properness]\label{006lemma5.5} For \((w, \alpha) \in \mathcal{U}\), the linearized operator \(\mathscr{F}_w(w, \alpha)\) is locally proper both from \(\mathcal{X}_{\mathrm{b}}\) to \(\mathcal{Y}_{\mathrm{b}}\) and from \(\mathcal{X}\) to \(\mathcal{Y}\).
\end{lemma}
As the argument for Lemma \ref{006lemma5.5} closely resembles that of \cite[Lemma 5.5]{SVHMHW2023}, we omit the details here.
\begin{lemma}[Invertibility]\label{006lemma5.6} For \((w, \alpha) \in \mathcal{U}\), the linear operator \(\mathscr{F}_w(0, \alpha)\) is invertible both from \(\mathcal{X}_{\mathrm{b}}\) to \(\mathcal{Y}_{\mathrm{b}}\) and from \(\mathcal{X}\) to \(\mathcal{Y}\).
\end{lemma}
\begin{proof}
Lemmas \ref{006lemma5.4} and \ref{006lemma5.5} establish that \(\mathscr{F}_w(0, \alpha)\) has a trivial kernel and is locally proper as a mapping from both \(\mathcal{X}_{\mathrm{b}} \to \mathcal{Y}_{\mathrm{b}}\) and \(\mathcal{X} \to \mathcal{Y}\). To show invertibility on \(\mathcal{X}_{\mathrm{b}} \to \mathcal{Y}_{\mathrm{b}}\), it remains to verify that the operator has Fredholm index zero.\par
To this end, we consider a one-parameter family of linear operators \(\mathscr{L}(t): \mathcal{X}_{\mathrm{b}} \to \mathcal{Y}_{\mathrm{b}}\), defined by
\[
\mathscr{L}(t) \dot{w} := \begin{pmatrix}
\gamma \dot{w}_{1} + \dot{w}_{2} \\
2\big(\dot{w}_{2y} + (\gamma -1- \epsilon_1)\dot{w}_{1y} + (\gamma + \alpha) t \dot{w}_{1} \big)
\end{pmatrix}\text{ for } t\in[0,1].
\]
A modification of the argument in Lemma \ref{006lemma5.4} shows that \(\mathscr{L}(t)\) is injective from \(\mathcal{X}_{\mathrm{b}} \to \mathcal{Y}_{\mathrm{b}}\), provided \(\alpha \in (0, \alpha_{\mathrm{cr}})\). In particular, the dispersion relation \eqref{0065.9} takes the form
$$
t(\gamma + \alpha) = (1 + \epsilon_1) k \coth(k),
$$
which implies that the kernel is trivial whenever $t(\gamma + \alpha) < 1 + \epsilon_1$.\par
Using the argument from Lemma \ref{006lemma5.5}, we deduce that $\mathscr{L}(t): \mathcal{X}_{\mathrm{b}} \to \mathcal{Y}_{\mathrm{b}}$ is locally proper and thus semi-Fredholm with a trivial kernel for all $t \in [0,1]$. Consequently, the Fredholm index of $\mathscr{L}(t)$ remains constant throughout the interval. The result of \cite[Lemma A.8]{RMCSWMHW2024} extends to our setting, implying that $\mathscr{L}(0)$ is invertible and therefore has Fredholm index zero. By the continuity of the Fredholm index, we deduce that \(\mathscr{L}(1) = \mathscr{F}_w(0, \alpha)\) also has index zero. Together with local properness and a trivial kernel, this implies that $\mathscr{F}_w(0, \alpha): \mathcal{X}_{\mathrm{b}} \to \mathcal{Y}_{\mathrm{b}}$ is invertible.\par
It remains to show that \(\mathscr{F}_w(0, \alpha)\) is also invertible as a map from \(\mathcal{X} \to \mathcal{Y}\). Since \(\mathcal{X} \subset \mathcal{X}_{\mathrm{b}}\), injectivity on \(\mathcal{X}\) follows immediately. To prove surjectivity, let \(\dot{f} \in \mathcal{Y}\). Because the operator is invertible on \(\mathcal{X}_{\mathrm{b}}\), there exists \(\dot{w} \in \mathcal{X}_{\mathrm{b}}\) such that \(\mathscr{F}_w(0, \alpha)\dot{w} = \dot{f}\). A translation argument ensures that \(\dot{w} \in \mathcal{X}\), completing the proof.
\end{proof}
\begin{lemma} \label{006lemma5.6+i} For \((w, \alpha) \in \mathcal{U}\), the linear operator \(\mathscr{F}_w(w, \alpha)\) is Fredholm with index \(0\) as a map from \(\mathcal{X}\) to \(\mathcal{Y}\).
\end{lemma}
\begin{proof}
Fix \((w, \alpha) \in \mathcal{U}\). By Remark \ref{rmk1}, the coefficients of \(\mathscr{F}_w(w, \alpha)\) approach those of \(\mathscr{F}_w(0, \alpha)\) as \(|x| \to \infty\). The proof is completed by combining Lemma \ref{006lemma5.6} with \cite[Lemmas A.12 and A.13]{MHW20151}.
\end{proof}

\section{Existence theory}\label{025section4}
\subsection{Small-amplitude}
Our objective in this subsection is to construct small-amplitude solutions for values of $\alpha$ near $\alpha_{\mathrm{cr}}$. Accordingly, we define
\[
\alpha = \alpha^\varepsilon := \alpha_{\mathrm{cr}}- \varepsilon = 1 - \gamma +\epsilon_1- \varepsilon,
\]
where \(\varepsilon > 0\) is a small parameter. The nodal property of solutions to \eqref{Non-dimensionalization equations 1} is
\begin{align}
\eta_{x} < 0\quad\text{in }(\mathcal{R} \cup \Gamma) \cap \{x > 0\},\quad \eta_{x} > 0\quad\text{in }(\mathcal{R} \cup \Gamma) \cap \{x < 0\},\label{0064.1}
\end{align}
where $\eta_{x} = \frac{\mathrm{d}Y}{\mathrm{d}x}$. We now state the main result of this subsection as the following theorem.
\begin{theorem}\label{006theorem7.1}
There exists \(\varepsilon_* > 0\) and a continuous local curve
\begin{align}
\mathscr{C}_{\mathrm{loc}} = \{(w^\varepsilon, \alpha^\varepsilon) : 0 < \varepsilon < \varepsilon_*\} \subset \mathcal{X} \times \mathbb{R} \label{006loc}
\end{align}
consisting of nontrivial symmetric solutions to \(\mathscr{F}(w, \alpha^\varepsilon) = 0\), with the asymptotic expansion
\[
w_1^\varepsilon(x, 1) = \frac{3\varepsilon}{3 - 3\gamma + \gamma^2 + \epsilon_1}
 \operatorname{sech}^2\left(\sqrt{\frac{3\varepsilon}{4(1 + \epsilon_1)}}x\right) + O(\varepsilon^{2 + \frac{1}{2}})
\]
in \(C_{\mathrm{b}}^{3+\beta}\). Moreover, the following properties hold:
\begin{itemize}
  \item[(i)] (Monotonicity): Every solution on \(\mathscr{C}_{\mathrm{loc}}\) satisfies the nodal property \eqref{0064.1}.
  \item[(ii)] (Uniqueness): If \(w \in \mathcal{X}\) and \(\varepsilon > 0\) are sufficiently small, and \(w\) satisfies \eqref{0064.1}, then the equation \(\mathscr{F}(w, \alpha^\varepsilon) = 0\) implies \(w = w^\varepsilon\).
  \item[(iii)] (Invertibility): The linearized operator \(\mathscr{F}_w(w^\varepsilon, \alpha^\varepsilon)\) is invertible as a map from \(\mathcal{X}\) to \(\mathcal{Y}\) for all \(0 < \varepsilon < \varepsilon_*\).
\end{itemize}
\end{theorem}
In the proof of Theorem \ref{006theorem7.1}, we employ the center manifold reduction method developed in \cite{RMCSWMHW2022}. This approach refines and extends the classical framework introduced by Kirchg\"{a}ssner \cite{KK1982} and further developed by Mielke \cite{AM1986, AM1988}. A principal advantage of this methodology is that the entire analysis is formulated in H\"{o}lder spaces, thereby allowing the reduced equations on the center manifold to be solved explicitly via a direct power series expansion.\par
As is standard in center manifold theory, we first expand the function space to allow for small exponential growth in the spatial variable $x$. To this end, we define exponentially weighted variants of the spaces $\mathcal{X}$ and $\mathcal{Y}$, given by
\[
\begin{aligned}
    &\mathcal{X}_{\mu} := \Big\{(w_{1}, w_{2}) \in (C_{\mu}^{3+\beta}(\overline{\mathcal{R}}))^{2}: \Delta w_{i} = 0 \; \text{in} \; \mathcal{R}, \; w_{i} = 0 \; \text{on} \; \mathcal{B} \Big\}, \\
    &\mathcal{Y}_{\mu} := C_{\mu}^{3+\beta}(\Gamma) \times C_{\mu}^{2+\beta}(\Gamma).
\end{aligned}
\]\par
For any $\mu > 0$, the linearized operator around the trivial flow is given by
$$
\mathscr{L}:= \mathscr{F}_{w}(0, \alpha_{\mathrm{cr}}).
$$
As stated explicitly in \eqref{0253.4}, this operator is a mapping from $\mathcal{X}_{\mu}$ to $\mathcal{Y}_{\mu}$. For sufficiently small $\mu > 0$, it is straightforward to verify that the kernel of $\mathscr{L}$ is two-dimensional and is given by
\[
\ker \mathscr{L} = \left\{
\begin{pmatrix}
(A + Bx)\varphi_1(y) \\
(A + Bx)\varphi_2(y)
\end{pmatrix}: A, B \in \mathbb{R} \right\},
\]
where
\[
\varphi =
\begin{pmatrix}
\varphi_1(y) \\
\varphi_2(y)
\end{pmatrix}
=
\begin{pmatrix}
y \\
-\gamma y
\end{pmatrix}.
\]
\begin{theorem}\label{006theorem7.2} (Center manifold reduction)
There exists \(0<\mu\ll 1\), neighborhoods \(\textbf{U} \subset \mathcal{X} \times \mathbb{R}\) and \(\textbf{V} \subset \mathbb{R}^3\) and a $C^3$ coordinate map $\Upsilon = (\Upsilon^1(A, B, \varepsilon), \Upsilon^2(A, B, \varepsilon)): \mathbb{R}^3\rightarrow\mathcal{X}_{\mu}$ satisfying
\[\Upsilon(0, 0, \varepsilon) = \Upsilon_A(0, 0, 0) =\Upsilon_B(0, 0, 0) = 0~ \text{ for all } \varepsilon,\]
such that the following hold:\par
(i) Suppose that \((w, \varepsilon) \in \boldsymbol{U}\) with $\alpha=\alpha_{\mathrm{cr}}-\varepsilon$ solves \eqref{0253.1}. Then \(q(x) := w_1(x, 1)\) solves the second-order ODE
\begin{align}
q'' = f(q, q', \varepsilon),\label{0067.1}
\end{align}
where \(f: \mathbb{R}^3 \to \mathbb{R}\) is the \(C^3\) mapping defined as
\begin{align}
f(A, B, \varepsilon) := \left.\frac{d^2}{dx^2}\right|_{x=0} \Upsilon(A, B, \varepsilon)(x, 1),\label{0067.2}
\end{align}
and admits the Taylor expansion
\begin{align}
f(A, B, \varepsilon) =& \frac{3}{1 + \epsilon_1}\varepsilon A -\frac{3(3 - 3\gamma + \gamma^2 + \epsilon_1)}{2(1 + \epsilon_1)}A^2 \notag\\&+ O\big((|A| + |B|)(|A| + |B| + |\varepsilon|)^2\big).\label{0067.3}
\end{align}\par
(ii) Conversely, if \(q: \mathbb{R} \to \mathbb{R}\) solves the ODE \eqref{0067.1} and \((q(x), q'(x), \varepsilon) \in \textbf{V}\) for all \(x\), then \(q = w_1(\cdot, 1)\) for solution \((w, \varepsilon) \in \boldsymbol{U}\) of \eqref{0253.1}. Moreover, we write it as
\[
w_i(x + \tau, y) = q(x)\varphi_i(y) + q'(x)\tau\varphi_i(y) + \Upsilon^i(q(x), q'(x), \varepsilon)(\tau, y),
\]
for $i=1,2$ and all \(\tau \in \mathbb{R}\).
\end{theorem}
\begin{rmk}
It is straightforward to verify that equation \eqref{0253.1} is invariant under the reversal transformation \( w \mapsto w(-\cdot, \cdot) \). As a result, we have
\[
\Upsilon(A, B, \varepsilon)(x, 1) = \Upsilon(A, -B, \varepsilon)(-x, 1),
\]
which implies that the function \(f\) is even with respect to the variable \(B\).
\end{rmk}
\begin{proof}
Strictly speaking, the center manifold results of \cite{RMCSWMHW2022} are formulated for scalar problems and for a restricted class of ``diagonal'' elliptic systems. After eliminating the auxiliary variable $\vartheta$, our reduced formulation does not fall exactly into these categories. Nevertheless, the same disclaimer as in the first paragraph of \cite[proof of Theorem 7.2]{SVHMHW2023} applies here: the system obtained after discarding $\vartheta$ is sufficiently close in structure to those covered by the theory, so that the arguments of \cite{RMCSWMHW2022} may be applied without essential modification. We therefore rely on the center manifold reduction in this broader sense.\par

Based on the results in \cite{RMCSWMHW2022}, it is sufficient to establish the validity of \eqref{0067.3}. Indeed, it has been established in \cite[Theorem 1.6]{RMCSWMHW2022} that the coordinate map \(\Upsilon\) admits the Taylor expansion
\[
\Upsilon(A, B, \varepsilon) := \sum_{\mathcal{J}} \Upsilon_{i'j'k'} A^{i'} B^{j'} \varepsilon^{k'} + O\left((|A| + |B|)(|A| + |B| + |\varepsilon|)^2\right),
\]
where
\[
\mathcal{J} = \left\{(i', j', k') \in \mathbb{N}^3 : i' + 2j' + k' \leq 3,\ i' + j' + k' \geq 2,\ i' + j' \geq 1 \right\}.
\]
The requirement for a second-order expansion, and hence for \(\Upsilon\) to be \(C^3\), follows from the regularity assumption on the background flow stated in \eqref{Non-dimensionalization equation 2}.\par
Explicitly, the index set \(\mathcal{J}\) contains only the tuples
\[
\mathcal{J} = \{(2, 0, 0),\ (1, 0, 1)\}.
\]
Thus, we expand \(\Upsilon\) as
\[
\Upsilon(A, B, \varepsilon) = \Upsilon_{200} A^2 + \Upsilon_{101} \varepsilon A + O\left((|A| + |B|)(|A| + |B| + |\varepsilon|)^2\right),
\]
in the space \(\mathcal{X}_\mu\).\par
We now compute the coefficients \(\Upsilon_{i'j'k'}\). Substituting the Taylor expansion of \(w\),
\[
w_i = (A + Bx) \varphi_i + A^2 \Upsilon_{200}^i + \varepsilon A \Upsilon_{101}^i + O\left((|A| + |B|)(|A| + |B| + |\varepsilon|)^2\right)
\]
into equation \eqref{0253.1}, we evaluate the boundary conditions.\par
The first boundary condition, \(\mathscr{F}_1(w, \alpha) = 0\), reduces to
\[
A^2(\gamma \Upsilon_{200}^1 + \Upsilon_{200}^2 + \frac{1}{2}\gamma) + \varepsilon A(\gamma \Upsilon_{101}^1 + \Upsilon_{101}^2) + O\left((|A| + |B|)(|A| + |B| + |\varepsilon|)^2\right) = 0.
\]
The second boundary condition, \(\mathscr{F}_2(w, \alpha) = 0\), becomes
\[
\begin{aligned}
0 =\ & A^2\big((2 + 2\epsilon_1)\Upsilon_{200}^1 + (2\gamma - 2-2\epsilon_1)\partial_y \Upsilon_{200}^1 + 2\partial_y \Upsilon_{200}^2  + 3 - 2\gamma + \gamma^2 + \epsilon_1\big) \\
& + \varepsilon A\big((2 + 2\epsilon_1)\Upsilon_{101}^1 + (2\gamma - 2-2\epsilon_1)\partial_y \Upsilon_{101}^1 + 2\partial_y \Upsilon_{101}^2 - 2\big) \\
& + O\left((|A| + |B|)(|A| + |B| + |\varepsilon|)^2\right).
\end{aligned}
\]
\par
Grouping terms of the same order, we derive two linear problems
\[
\mathscr{L}\Upsilon_{200} = \begin{pmatrix} - \frac{1}{2}\gamma \\ -3 + 2\gamma - \gamma^2 - \epsilon_1 \end{pmatrix} \quad \text{and} \quad \mathscr{L}\Upsilon_{101} = \begin{pmatrix} 0 \\ 2 \end{pmatrix}.
\]
\par
For any \(s_1, s_2 \in \mathbb{R}\), the general problem
\[
\mathscr{L} \tilde{\Upsilon} = \begin{pmatrix} s_1 \\ s_2  \end{pmatrix}, ~\text{with } \tilde{\Upsilon}^1(0, 0) = \partial_x \tilde{\Upsilon}^1(0, 0) = 0,
\]
is solved by
\[
\begin{aligned}
\tilde{\Upsilon}^1 &= \frac{3}{4}\frac{s_2 - 2s_1}{1 + \epsilon_1} x^2 y - \frac{1}{4}\left(s_2 - 2s_1 - \epsilon_1 \frac{s_2 - 2s_1}{1 + \epsilon_1}\right) y(y^2 - 1), \\
\tilde{\Upsilon}^2 &= -\gamma \tilde{\Upsilon}^1 + s_1 y.
\end{aligned}
\]
By \cite{RMCSWMHW2022}, the solution is unique. Substituting \(s_1 = 0\), \(s_2 = 2\) yields \[\partial_x^2 \Upsilon_{101}^1(0, 0) = \frac{3}{1 + \epsilon_1}.\]
 Similarly, for \(s_1 = -\frac{1}{2}\gamma\) and \(s_2 = -3 + 2\gamma - \gamma^2 - \epsilon_1\), we obtain
\[
\partial_x^2 \Upsilon_{200}^1(0, 0) = -\frac{3(3 - 3\gamma + \gamma^2 + \epsilon_1)}{2(1 + \epsilon_1)}.
\]
Substituting these into \eqref{0067.2}, we obtain the desired expansion in \eqref{0067.3}.
\end{proof}
We now proceed to prove Theorem \ref{006theorem7.1}.
\begin{proof}[\textbf{Proof of Theorem \ref{006theorem7.1}}]
By Theorem \ref{006theorem7.2}, the analysis of system \eqref{Non-dimensionalization equations 1} can be reduced to studying the reduced system \eqref{0067.1}. To proceed, we introduce the scaled variables
\[
x = |\varepsilon|^{-\frac{1}{2}}X, \quad q(x) = \varepsilon Q(X), \quad q_x(x) = \varepsilon |\varepsilon|^{\frac{1}{2}} P(X).
\]
Under this transformation, the equation \eqref{0067.3} becomes
\begin{align}
Q_{XX} = P_X = \frac{3}{1 + \epsilon_1}Q -\frac{3(3 - 3\gamma + \gamma^2 + \epsilon_1)}{2(1 + \epsilon_1)} Q^2 + O\left(|\varepsilon|^{\frac{1}{2}}(|Q| + |P|)\right). \label{0067.4}
\end{align}
A standard computation shows that the explicit homoclinic solution:
\[
Q(X) = \frac{3}{3 - 3\gamma + \gamma^2 + \epsilon_1} \, \operatorname{sech}^2\big(\sqrt{\frac{3}{4(1 + \epsilon_1)}} X\big).
\]
This function solves \eqref{0067.4} and corresponds to a homoclinic orbit to the origin that intersects the \(Q\)-axis transversely at the point \((Q_0, 0)\), where
\[
Q_0 = \frac{3}{3 - 3\gamma + \gamma^2 + \epsilon_1}.
\]
See Figure \ref{006figurexiangweitu}.
\begin{figure}[htbp]
	\centering
	\includegraphics[
	width=18cm,
	height=9cm,
	trim=4.8cm 1.5cm 0cm 0,  
	clip
	]{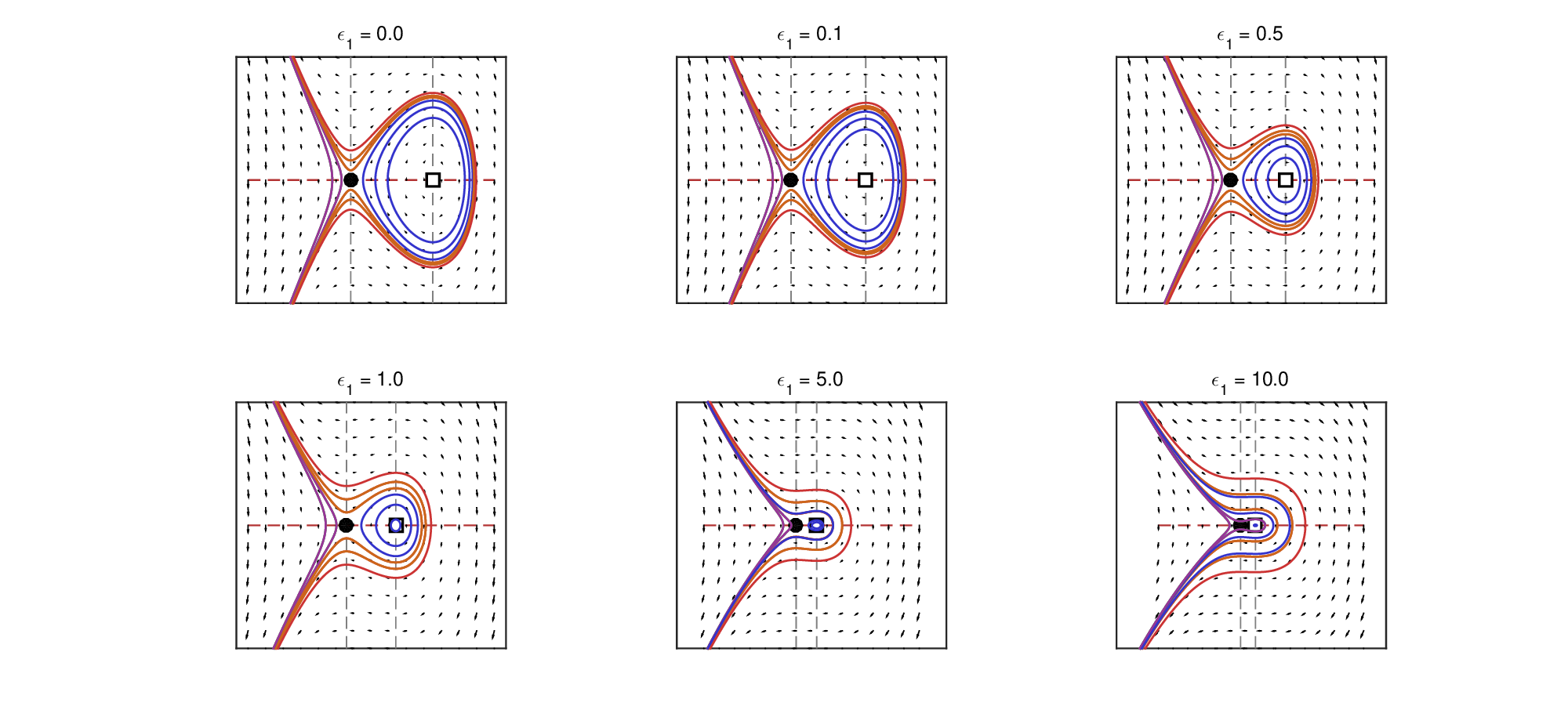}
	\caption{Phase portrait of ODE \eqref{0067.4} with \(\varepsilon=0, \gamma=1, \epsilon_1 \in \{0, 0.1, 0.5, 1, 5, 10\}\)}
	\label{006figurexiangweitu}
\end{figure}

The reversibility of \(\Upsilon\) ensures that equation \eqref{0067.4} is reversible. By standard theory for planar reversible systems, this homoclinic orbit persists for all sufficiently small \(\varepsilon\), giving rise to a continuous one-parameter family of homoclinic solutions. Reversing the scaling yields the local solution curve \(\mathscr{C}_{\mathrm{loc}}\), as stated in \eqref{006loc}.
\par
We now verify the three properties:\par
(i) The inequality $w_{1x}^\varepsilon > 0$ on $\Gamma \cap \{x < 0\}$ follows from the symmetry and monotonicity of the leading-order solution $w_1^\varepsilon$. Specifically, the reduced ODE together with the boundary conditions on $\Gamma$ ensures that $w_1^\varepsilon$ increases in $x$ for $x < 0$ and decreases for $x > 0$. The strong maximum principle implies
\[
w_{1x}^\varepsilon > 0 \quad \text{in } (\Gamma \cup \mathcal{R}) \cap \{x < 0\}.
\]
By symmetry, it follows that
\[
w_{1x}^\varepsilon < 0 \quad \text{in } (\Gamma \cup \mathcal{R}) \cap \{x > 0\}.
\]
Hence, property \eqref{0064.1} holds.
\par
(ii) Suppose \((w, \alpha^\varepsilon)\) is another solution to \(\mathscr{F}_w(w, \alpha^\varepsilon) = 0\). By center manifold theory, \(w\) must correspond to a homoclinic orbit of the reduced ODE. Since the wave is an elevation wave satisfying \eqref{0064.1}, the orbit lies entirely in the region where \(Q > 0\), and thus we conclude that \(w = w^\varepsilon\).
\par
(iii) From Lemma~\ref{006lemma5.6+i}, the operator \(\mathscr{F}_w(w, \alpha)\) is Fredholm of index zero. Therefore, \(\mathscr{F}_w(w, \alpha)\) is invertible if and only if its kernel is trivial. Due to translation invariance, we know
\[
\mathscr{F}_w(w^\varepsilon, \alpha^\varepsilon) w_x^\varepsilon = 0.
\]\par
Using \cite[Theorem 1.9]{RMCSWMHW2022}, any element \(\dot{w} \in \ker \mathscr{F}_w(w, \alpha^\varepsilon)\) satisfies the linearized reduced ODE
\[
\dot{q}'' = \nabla_{(q, q')} f(q, q', \varepsilon) \cdot (\dot{q}, \dot{q}').
\]
Applying the same scaling as before, this becomes the linear planar system
\[
\begin{pmatrix} \dot{Q} \\ \dot{P} \end{pmatrix}_X = \mathcal{M}(X) \begin{pmatrix} \dot{Q} \\ \dot{P} \end{pmatrix},
\]
where
\[
\lim_{X \to \pm\infty} \mathcal{M}(X) =
\begin{pmatrix}
0 & 1 \\
\frac{3}{1 + \epsilon_1} + O(\varepsilon^{1/2}) & O(\varepsilon^{1/2})
\end{pmatrix}.
\]
This matrix has one strictly positive and one strictly negative eigenvalue. Hence, the system admits at most one (up to scaling) bounded solution, which corresponds to \(w_x^\varepsilon\). Since \(w^\varepsilon\) is even in \(x\), \(w_x^\varepsilon\) is odd and therefore lies outside the space \(\mathcal{X}\). Thus, the kernel of \(\mathscr{F}_w(w^\varepsilon, \alpha^\varepsilon): \mathcal{X} \to \mathcal{Y}\) is trivial. This completes the proof.
\end{proof}

\begin{rmk}While the argument is closely related to the proof of Theorem 7.1 in \cite{SVHMHW2023}, the presence of the electric field and the associated scaling of the current require the modifications presented here to ensure that these properties hold in the present setting.
\end{rmk}
\subsection{Large-amplitude}
Let us first state a modified version of the global bifurcation theorem for solitary waves as follows.
\begin{theorem}\label{006theoremfulu2} (\cite[Theorem B.1]{SVHMHW2023})
Let $\mathscr{X}$ and $\mathscr{Y}$ be Banach spaces, $\mathscr{U}$ be an open subset of $\mathscr{X}\times\mathbb{R}$ with
$( 0, 0) \in \partial \mathscr{U}$. Consider a real-analytic mapping $\mathscr{F}: \mathscr{U} \to \mathscr{Y}$. Suppose that\par
(I) for any $(\mu,x)\in\mathscr{U}$ with $\mathcal{F}(\mu,x)=0$ the Fr\'{e}chet derivative $\mathcal{F}_x(\mu,x):\mathscr{X}\to\mathscr{Y}$ is locally proper;\par
(II) there exists a continuous curve $\mathscr{C}_\mathrm{loc}$ of solutions to $\mathcal{F}(\mu, x) = 0$, parameterized as
$$\mathscr{C}_{\mathrm{loc}}:=\{(\mu,\tilde{x}(\mu)):0<\mu<\mu_{*}\}\subset\mathcal{F}^{-1}(0),$$
for some $\mu_*>0$ and continuous $\tilde{x}$ with values in $\mathscr{X}$ and $\lim_{\mu\searrow0}\tilde{x}(\mu)=0$;\par
(III) the linearized operator $\mathcal{F}_x(\mu,\tilde{x}(\mu)):\mathscr{X}\to\mathscr{Y}$ is invertible for all $\mu$.\par
Then $\mathscr{C}_\mathrm{loc}$ is contained is a curve of solutions $\mathscr{C}$, parameterized as
$$\mathscr{C}:=\{(\mu(s),x(s)):0<s<\infty\}\subset\mathcal{F}^{-1}(0)$$
for some continuous $(0,\infty)\ni s\mapsto(x(s),\mu(s))\in\mathscr{U}$, with the following properties\par
(a) One of the following alternatives holds:\par
\quad(i) (Blow-up) As $s\rightarrow \infty$,
$$N(s):=\|x(s)\|_{\mathcal{X}}+\frac{1}{\mathrm{dist}((\mu(s),x(s)),\:\partial\mathcal{U})}+\mu(s)\to\infty.$$\par
\quad(ii) (Loss of compactness) There exists a sequence $s_n\to\infty$  such that $\mathrm{sup}_n N(s_n)<\infty$\par
\quad\quad \quad but $\left\{x(s_{n})\right\}$ has no subsequences converging in $\mathscr{X}$.\par
(b) Near each point $(\mu(s_0),x(s_0))\in\mathscr{C}$, we can reparameterize $\mathscr{C}$ so that $s\mapsto(\mu(s),x(s))$
is real analytic.\par
(c) $( \mu ( s) , x( s) ) \notin \mathscr{C} _{loc}$ for $s$ sufficiently large.
\end{theorem}
Then, we can establish the following result by Theorem \ref{006theoremfulu2}, which extends the local solution curve $\mathscr{C}_{\mathrm{loc}}$ to a global continuum.\par
\begin{theorem}[Global continuation]\label{006theorem7.4} The local curve $\mathscr{C}_\mathrm{loc}$ is contained in a continuous curve of solutions parameterized as
\[
\mathscr{C} = \{(w(s), \alpha(s)) : 0 < s < \infty\} \subset \mathcal{U},
\]
with the following properties:
\begin{itemize}
\item[(a)] One of the following alternatives holds:\par
 \quad (i) (Blow-up) As $s \to \infty$,
 \begin{align}
   N(s) := \|w(s)\|_{\mathcal{X}}+\frac{1}{\kappa(w(s), \alpha(s))-\epsilon_1} + \frac{1}{\lambda(w(s), \alpha(s))} + \frac{1}{\alpha(s)} + \frac{1}{\alpha_{\mathrm{cr}} - \alpha(s)} \longrightarrow \infty.
  \label{0067.4+1}
\end{align}
\quad (ii) (Loss of compactness) There exists a sequence $\{s_n\} \to \infty$ as $n\to \infty$ such that
$\sup_n N(s_n) < \infty$, but $\{w(s_n)\}$ does not have subsequences converging in $\chi$.
\item[(b)] Near each point $(w(s_0), \alpha(s_0)) \in \mathscr{C}$, we can re-parameterize $\mathscr{C}$ so that the mapping $s \mapsto (w(s), \alpha(s))$ is real analytic.
\item[(c)] For $s$ sufficiently large, $(w(s), \alpha(s)) \notin \mathscr{C}_\mathrm{loc}$.
\end{itemize}
\end{theorem}
\begin{proof}
It is clear that \(\mathscr{F}\) is real analytic on the open set \(\mathcal{U}\). According to Lemma \ref{006lemma5.5}, the linearized operator \(\mathscr{F}_w(w, \alpha)\) is locally proper for all \((w, \alpha) \in \mathcal{U}\). In Theorem \ref{006theorem7.1}, we constructed a local curve of solutions \(\mathscr{C}_{\mathrm{loc}} = \{(w^\varepsilon, \alpha^\varepsilon): 0 < \varepsilon < \infty\} \subset \mathcal{U}\). Furthermore, part (iii) of Theorem \ref{006theorem7.1} guarantees that the operator \(\mathscr{F}_w(w, \alpha)\) is invertible along \(\mathscr{C}_{\mathrm{loc}}\). Applying Theorem \ref{006theoremfulu2}, we conclude the proof of Theorem \ref{006theorem7.4}.
\end{proof}
\section{Global bifurcation structure}\label{025section5}
In this section, we analyze several alternatives presented in Theorem \ref{006theorem7.4} and conclude by establishing Theorem \ref{006theorem7.9ABOVE}.
\subsection{Nodal properties}\label{025subsection5.1}
The monotonicity property plays a crucial role in ruling out alternative (ii) in Theorem \ref{006theorem7.4}. However, the set of monotone functions is neither open nor closed in the relevant topology. To overcome this difficulty, we introduce nodal properties-specifically, sign conditions on the derivatives of the solutions. These conditions not only imply monotonicity but also define a set that is both open and closed in the appropriate topological space. \par
Compared with \cite{SVHMHW2023}, the nodal properties analyzed in this paper are more intricate due to the presence of a coupled electric field. The analysis requires accounting not only for the properties of the velocity field but also for the behavior of the electric field.\par

Without loss of generality, we restrict our attention to the right half of the domain \(\mathcal{R}\). We define
\[
\mathcal{R}^{+} := \{(x, y) \in \mathcal{R} : x > 0\}, \, \Gamma^{+} := \{(x, y) \in \Gamma : x > 0\} \text{ and } L := \{(x, y) \in \mathcal{R} : x = 0\}.
\]
The monotonicity condition with respect to $v$ is given by
\begin{align}
v < 0 \quad \text{in } \Gamma^+ \cup \mathcal{R}^+.\label{0064.2}
\end{align}
If \eqref{0064.2} holds, then $\eta_x < 0 $ in $\Gamma^+ \cup \mathcal{R}^+$, as shown in the following lemma. Moreover, from \eqref{Velocity field and electric field}, it follows that $e_1 > 0$ in $ \Gamma^+ \cup \mathcal{R}^+$.
\begin{lemma}\label{006lemma4.1}
Let \((\eta, \zeta, \vartheta, \alpha)\) be a solution to \eqref{Non-dimensionalization equations 1}, and define \(v\) as in \eqref{Velocity field and electric field}. If \eqref{0064.2} holds, we have $\eta_x < 0$ in $\Gamma^+ \cup \mathcal{R}^+$.
\end{lemma}
\begin{proof}
By the kinematic boundary condition \eqref{Velocity field and electric field equation e}, the vector fields $(u, v)$ and
$(\eta_y, \eta_x)$ are non-vanishing and parallel when restricted to $\Gamma$.  Moreover, the asymptotic condition \eqref{w equation 2} implies that $\eta_y \to 1$ as $x \to \infty$, while \eqref{Velocity field and electric field equations 2} guarantees that $u$ converges to $1$ in the same limit. Thus $u$ and $\eta_y$ have the same sign, which implies that $v$ and $\eta_x$ have the same sign. It follows from \eqref{0064.2} that $\eta_x<0$ on $\Gamma^+$.

Next, differentiating \eqref{Non-dimensionalization equation a} with respect to \(x\) yields \(\Delta \eta_x = \partial_x (\Delta \eta) = 0\). Applying the strong maximum principle to \(\eta_x\) in \(\mathcal{R}^+\), and using the boundary condition established above, we conclude that \(\eta_x < 0\) throughout \(\mathcal{R}^+\).
\end{proof}
To demonstrate that monotonicity is preserved along the global solution curve $\mathscr{C}$, we show that condition \eqref{0064.1} defines a subset of nontrivial solutions to \eqref{Non-dimensionalization equations 1} that is both relatively open and relatively closed.
\begin{lemma}[Closed property]\label{006lemma4.2}
Let $(\eta, \zeta, \vartheta, \alpha)$ be a solution to \eqref{Non-dimensionalization equations 1}, and define $v$ as in \eqref{Velocity field and electric field}. If $v \leq 0$ on $\Gamma^+$, we have either \eqref{0064.2} holds or $v \equiv 0$.
\end{lemma}
\begin{proof}
It suffices to consider $v \not\equiv 0$. From conditions \eqref{Velocity field and electric field equation h} and \eqref{Velocity field and electric field equations 2}, we know that $v = 0$ along $\mathcal{B}$ as $x \to \infty$. In addition, \eqref{Velocity field and electric field equations 4} implies that $v = 0$ on $L$. Since $v\leq0$ on $\Gamma^+$, the strong maximum principle yields $v < 0$ in $\mathcal{R}^+$. Then, we have $e_1>0$ in $\mathcal{R}^+$.\par
It remains to show $v < 0$ also holds on $\Gamma^+$. To this end, we differentiate the dynamic boundary condition \eqref{Velocity field and electric field equation g} with respect to $x$, and multiply by $-\frac{u}{2}$, which yields
\begin{align}
- u(vv_x + uu_x +\epsilon_1{e_2}{e_2}_x +\epsilon_1{e_1}{e_1}_x + \alpha \eta_x) = 0. \label{0064.4}
\end{align}
Using \eqref{Velocity field and electric field equation a}, \eqref{Velocity field and electric field equation b} and \eqref{Velocity field and electric field equation e}, we substitute \(u_x = \gamma \eta_x - v_y\), \({e_2}_x ={e_1}_y\) and \(u\eta_x=v\eta_y\) into \eqref{0064.4} to obtain
\begin{align}
u(uv_y - vv_x-\epsilon_1{e_2}{e_1}_y -\epsilon_1{e_1}{e_1}_x)- \eta_y v(\gamma u + \alpha) = 0. \label{0064.5}
\end{align}
 As a result, $v$ and $e_1$ satisfy the following boundary value problem
\begin{subequations}\label{0064.3}
\begin{align}
\Delta v &= 0 \quad \text{in } \mathcal{R}, \label{0064.3a} \\
\Delta e_1 &= 0 \quad \text{in } \mathcal{R}, \label{0064.3b} \\
u(uv_y - vv_x-\epsilon_1{e_2}{e_1}_y -\epsilon_1{e_1}{e_1}_x)- \eta_y v(\gamma u+ \alpha)&= 0 \quad \text{on } \Gamma, \label{0064.3c} \\
v &= 0 \quad \text{on } \mathcal{B}, \label{0064.3d}\\
e_1 &= 0 \quad \text{on } \mathcal{B}, \label{0064.3e}
\end{align}
\end{subequations}
where the boundary condition \eqref{0064.3c} follows from \eqref{0064.5}.\par

We argue by contradiction and suppose that $v$ attains its maximum value of 0 at some point $(x_0, 1) \in \Gamma^+$. Remark \ref{025remk2.1} gives that $u \neq 0$ at $(x_0, 1)$. It then follows from \eqref{Velocity field and electric field equation e} and \eqref{Velocity field and electric field equation f} that $\eta_x = 0$ and $\eta_y \neq 0$ at this point. Furthermore, the relationship \eqref{Velocity field and electric field} implies $e_1 = 0$ and $e_2 \neq 0$ there. Substituting these into \eqref{0064.3c} yields $u^2 v_y - \epsilon_1 u e_2 e_{1y} = 0$ at $(x_0, 1)$. We now claim that $u$ and $e_2$ have the same sign at $(x_0, 1)$. By the kinematic boundary condition \eqref{Velocity field and electric field equation e}, the vector fields $(u, v)$ and
$(\eta_y, \eta_x)$ are non-vanishing and parallel when restricted to $\Gamma$. By \eqref{Velocity field and electric field}, then the vector fields $(u, v)$ and
$(e_2, -e_1)$ are non-vanishing and parallel when restricted to $\Gamma$. That is, there exists a continuous function $c(x)$ such that $(u(x,1),v(x,1))=c(x)(e_2(x,1), -e_1(x,1))$. In addition, the asymptotic condition \eqref{Velocity field and electric field equations 2} shows that $u, e_2\rightarrow 1$ on $\Gamma^+$ as $x\rightarrow\infty$, which means that $c(x)\rightarrow1>0$ as $x\rightarrow\infty$. Indeed, if we assume that $u$ and $e_2$ have the opposite signs at $(x_0, 1)$, then we have $c(x_0)<0$. Since $c(x)$ is continuous with respect to $x$, there exists $x_1\in (x_0, +\infty)$ such that $c(x_1)=0$.
We conclude that at the point $(x_1,1)$, both $u = 0$ and $v = 0$, which contradicts Remark \ref{025remk2.1}. Hence $u$ and $e_2$ must share the same sign at $(x_0, 1)$. Then, by the Hopf boundary point lemma, we obtain  $v_y > 0$ and $ e_{1y} < 0$ at $(x_0, 1)$, which leads to a contradiction.\par

Consequently, the strict inequality $v < 0$ must hold on $\Gamma^+$. This completes the proof.
\end{proof}
\begin{lemma}[Open property]\label{006lemma4.3}
Let \((\eta^{*}, \zeta^{*}, \vartheta^{*}, \alpha^{*})\) be a solution to \eqref{Non-dimensionalization equations 1}, and let \(v^*\) be defined by \eqref{Velocity field and electric field}. Suppose that \eqref{0064.2} holds, \(\alpha^* < \alpha_{\mathrm{cr}}\). Then there exists \(\varepsilon > 0\) such that for any solution \((\eta, \zeta, \vartheta, \alpha)\) to \eqref{Non-dimensionalization equations 1} satisfying
\[
\|\eta - \eta^{*} \|_{C^3(\mathcal{R})} + \|\zeta - \zeta^{*} \|_{C^3(\mathcal{R})} + \|\vartheta - \vartheta^{*} \|_{C^3(\mathcal{R})} + |\alpha - \alpha^{*}| < \varepsilon,
\]
the corresponding \(v\) satisfies \eqref{0064.2}.
\end{lemma}
We divide the domain $\mathcal{R}^+$ into a bounded region and a semi-infinite strip. For any $M > 0$, we define the semi-infinite strip by
\[
\mathcal{R}_M^+ := \{(x, y) \in \mathcal{R} : x > M\},
\]
with
\begin{align}
\Gamma_M^+ &:= \{(x, y) \in \Gamma : x > M\}, \notag\\
\mathcal{B}_M^+ &:= \{(x, y) \in \mathcal{B} : x > M\}, \notag\\
L_M^+ &:= \{(x, y) \in \mathcal{R} : x = M\}.\notag
\end{align}
The structure of the proof follows the approach presented in \cite[Proposition 4.3]{SVHMHW2023}. We now consider the two regions separately, beginning with the bounded rectangular region.
\begin{lemma}\label{006lemma4.5}
Let $(\eta^{*}, \zeta^{*}, \vartheta^{*}, \alpha^{*})$ be a solution to \eqref{Non-dimensionalization equations 1}, and let $(u^*, v^*)$ be defined as in \eqref{Velocity field and electric field}. Suppose that \eqref{0064.2} holds, \(\alpha^{*} < \alpha_{\mathrm{cr}}\). Then, for any \(M > 0\), there exists a constant \(\varepsilon_M > 0\) such that for every solution \((\eta, \zeta, \vartheta, \alpha)\) to \eqref{Non-dimensionalization equations 1} satisfying
\[
\|\eta - \eta^{*} \|_{C^3(\mathcal{R})} + \|\zeta - \zeta^{*} \|_{C^3(\mathcal{R})} + \|\vartheta - \vartheta^{*} \|_{C^3(\mathcal{R})} + |\alpha - \alpha^{*}| < \varepsilon_M,
\]
the corresponding \(v\) satisfies
\begin{align}
v < 0 \quad \text{in } (\mathcal{R} \cup \Gamma) \cap \{0 < x \leq 2M\}.\label{0064.6}
\end{align}
\end{lemma}
\begin{proof}
Let $\eta^*, \zeta^*, \vartheta^*, \alpha^*, u^*$ and $v^*$ be as given in the statement of the lemma. We begin by considering the following inequalities:
\begin{subequations}\label{0064.7}
\begin{align}
v^*_x &< 0 && \text{on } L^+ \cup \{(0,1)\}, \label{0064.7a} \\
v^*_y &< 0 && \text{on } \mathcal{B}^+ \cap \{0 < x \leq 2M\}, \label{0064.7b} \\
v^*_{xy} &< 0 && \text{at } (0,0).\label{0064.7c}
\end{align}
\end{subequations}
We claim that these inequalities follow directly from conditions \eqref{0064.2} and \eqref{Non-dimensionalization equations 1}. Since \(v^*\) and \({e_1^*}\) are harmonic and odd in \(x\), we have
\[
v^* = v^*_{xx} = v^*_{yy} = v^*_{y}={e_1^*} = {e_1^*}_{xx} = {e_1^*}_{yy} = {e_1^*}_{y} = 0 \quad \text{at } (0,1).
\]
Since \(u^*\) and \({e_2^*}\) are even in \(x\), we have $u^*_x={e_2^*}_x=0$ at $(0,1)$. To verify \eqref{0064.7a} and \eqref{0064.7b}, we differentiate \eqref{0064.3c} with respect to $x$, and then evaluate the resulting identity at the wave crest $(0,1)$
\begin{align}
(u^*)^2 v^*_{xy} - u^*(v^*_x)^2-\epsilon_1 u^*{e_2^*}{e_1^*}_{xy}-\epsilon_1 u^*{e_1^*}_x{e_1^*}_{x}- \eta^*_y(\gamma u^* + \alpha^*)v^*_x = 0 \quad \text{at } (0,1).\label{0064.8}
\end{align}
From Remark \ref{025remk2.1}, we have $u^* \neq 0$ at $(0, 1)$. Since $(u^*, v^*)$ and $(\eta_y^*, \eta_x^*)$ are parallel, it follows that $\eta_x^* = 0$ and $\eta_y^* \neq 0$ there. Consequently, \eqref{Velocity field and electric field} implies $e_2^* \neq 0$ at $(0, 1)$. Similarly, following the reasoning in the proof of Lemma \ref{006lemma4.2}, there exists a continuous function $C(x)$ such that
$$e_1^*(x,1)=C(x)v^*(x,1).$$
Differentiating this relation with respect to $x$ and evaluating at $(0,1)$ gives
$$e^*_{1x}(0,1)=C_x(0)v^*(0,1)+C(0)v^*_x(0,1).$$
Assuming $v^*_x(0,1)= 0$ and using the fact that $v^*(0,1)=0$, we obtain $e^*_{1x}(0,1)=0$. Substituting these into equation \eqref{0064.8} yields
$$({u^*})^2 v^*_{xy} - \epsilon_1 u^* e^*_2 e^*_{1xy} = 0.$$
Then, as argued in the proof of Lemma \ref{006lemma4.2}, we conclude that $u^*$ and $e^*_2$ have the same sign at $(0, 1)$. By the Serrin's edge point lemma, we conclude that $v^*_{xy} > 0$ and $e^*_{1xy} < 0$ at $(0, 1)$, which leads to a contradiction. Thus there holds that $v^*_x(0,1)<0$.\par

Moreover, since $v^* = 0$ along $L^+$, the Hopf lemma implies that $v^*_x < 0$ on $L^+$. Likewise, because $v^*$ vanishes along $\mathcal{B}^+$, the Hopf lemma guarantees that $v^*_y < 0$ on $\mathcal{B}^+ \cap \{0 < x \leq 2M\}$. These arguments establish \eqref{0064.7a} and \eqref{0064.7b}.\par
We prove \eqref{0064.7c} at the corner point $(0,0)$. As previously noted, we have $v^* = v^*_{xx} = v^*_y = v^*_{yy} ={e_1}^* = {e_1}^*_{xx} = {e_1}^*_y = {e_1}^*_{yy} = 0$ at this point. From \eqref{0064.7b}, we have $v^*_{xy} \leq 0$ there. However, since $v^*_x = 0$ along $\mathcal{B}$, $v^*_{xy} = 0$ at $(0,0)$ would once again contradict Serrin's edge point lemma. Therefore, we must conclude that $v^*_{xy} < 0$ at $(0,0)$.\par

Following the approach in \cite{ACWS2007}, one can show that the combined \eqref{0064.6} and \eqref{0064.7} define an open subset of $C^2(\mathcal{R}^+ \cap \{0 \leq x \leq 2M\})$. Hence, for sufficiently small $\varepsilon_M > 0$, any solution sufficiently close to $v$ will also satisfy \eqref{0064.6} and \eqref{0064.7}, thus completing the proof.
\end{proof}
We now turn our attention to the semi-infinite strip $\mathcal{R}_M^+$:
\begin{lemma}\label{006lemma4.6}
Fix \(\alpha_0 \in (0, \alpha_{\mathrm{cr}})\). Then there exists a constant \(\delta = \delta(\alpha_0) > 0\) such that the following holds. Let \((\eta, \zeta, \vartheta, \alpha)\) be a solution to \eqref{Non-dimensionalization equations 1} with \(0 < \alpha \leq \alpha_0\), and define \(v\) and \(w\) by \eqref{Velocity field and electric field} and \eqref{w equations 1}, respectively. Suppose that for some $M > 0$,
$$
\|w\|_{C^1(\mathcal{R}_M^+)} < \delta, \quad v \leq 0 \quad \text{on } L_M^+.
$$
Then either $v < 0$ in $\mathcal{R}_M^+ \cup \Gamma_M^+$, or $v \equiv 0$.
\end{lemma}
\begin{proof}
Choose \(\varepsilon, \delta > 0\) sufficiently small so that
\begin{align}
b:=\frac{u^2+\epsilon_1e_2^2- \eta_y (\gamma u + \alpha)(1 + \varepsilon)}{u^2}> 0 \quad \text{and} \quad u>0, {e_2} > 0 \quad \text{on } \Gamma_M^+. \label{0064.9}
\end{align}
As \(\varepsilon, \delta \to 0\), we have \(u, {e_2} \to 1\) and
\(b\to 1 - \gamma+\epsilon_1 - \alpha\) uniformly on \(\Gamma_M^+\).
Now define the auxiliary functions
\[
f := \frac{v}{y + \varepsilon} \quad\text{and} \quad g := \frac{-{e_1}}{y + \varepsilon}.
\]
It is obvious that $f$ and $v$ share the same sign, and similarly, $g$ and $-e_1$ share the same sign. Furthermore, $f, g\rightarrow 0$ as $x \to \infty$, as implied by \eqref{Velocity field and electric field equations 2}. A straightforward computation based on \eqref{0064.3} demonstrates that $f$ and $g$ satisfy the following elliptic problem
\begin{subequations}\label{0064.10}
\begin{align}
\Delta f + \frac{2}{y + \varepsilon} f_y &= 0 &&\text{in } \mathcal{R}_M^+, \label{0064.10a} \\
\Delta g + \frac{2}{y + \varepsilon} g_y &= 0 &&\text{in } \mathcal{R}_M^+, \label{0064.10b} \\
(1 + \varepsilon)\big(u^2 f_y - u v f_x +\epsilon_1ue_2g_y+\epsilon_1ue_1g_x\big)+ b f&= 0 &&\text{on } \Gamma_M^+, \label{0064.10c} \\
f &= 0 &&\text{on } \mathcal{B}_M^+, \label{0064.10d}\\
g &= 0 &&\text{on } \mathcal{B}_M^+.\label{0064.10e}
\end{align}
\end{subequations}
By \eqref{0064.9}, the coefficients of \(f\) and \(f_y\) in \eqref{0064.10c} are strictly positive. Assume, to the contrary, that $f \not\equiv 0$ and attains its nonnegative maximum at some point $(x_0, y_0) \in \mathcal{R}_M^+ \cup \Gamma_M^+$. The strong maximum principle implies that this maximum must occur on $\Gamma_M^+$, where the Hopf lemma gives $f_y(x_0, y_0) > 0$.

Similarly, following the reasoning in the proof of Lemma \ref{006lemma4.2}, there exists a continuous function $c(x)$ such that
$(u(x,1),v(x,1)) = c(x)(e_2(x,1), -e_1(x,1))$ for $x > M$. It follows that $f(x)=c(x)g(x)$ along the streamline. By Remark \ref{025remk2.1}, $u$ and $v$ do not vanish simultaneously on $\Gamma_M^+$, which implies $c(x) \neq 0$ on $\Gamma_M^+$. Furthermore, the asymptotic condition \eqref{Velocity field and electric field equations 2} ensures $c(x) > 0$.

Direct calculation yields $g_x = \frac{f_x c(x) - f(x) c'(x)}{(c(x))^2}$, which simplifies to $g_x(x_0, 1) = \frac{-f(x_0, 1) c'(x_0)}{(c(x_0))^2}$. The boundary condition \eqref{0064.10c} at $(x_0, 1)$ is given by
\begin{equation}\label{e510}
	(1+\varepsilon)u^2 f_y + (1+\varepsilon)\epsilon_1 u e_2 g_y + \left( b - (1+\varepsilon)\epsilon_1 u e_1 \frac{c'(x_0)}{c(x_0)^2} \right) f = 0.
\end{equation}
From Theorem \ref{006theorem3.8} (proved in Section \ref{025subsection5.4}), there exists $\alpha_0$ such that $\alpha \le \alpha_0 < \alpha_{\mathrm{cr}} = 1 - \gamma + \epsilon_1$. This guarantees that the zero-order coefficient $b$ satisfies $b \geq \alpha_{\mathrm{cr}} - \alpha_0 > 0$. Since $e_1 \to 0$ and $c'(x) \to 0$ as $x \to \infty$, there exists a constant $M_1 > 0$ such that for all $x_0 > M_1$, $(1+\varepsilon)\epsilon_1 u e_1 \frac{c'(x_0)}{c(x_0)^2}$ is strictly bounded by $\frac{1}{2}(\alpha_{\mathrm{cr}} - \alpha_0)$. Consequently, for any $M \ge M_1$, the coefficient of $f$ in \eqref{e510} is strictly positive.

On the other hand, we claim that $g_y(x_0, 1) \ge 0$. Note that $g$ satisfies \eqref{0064.10b} in $\mathcal{R}_M^+$ and vanishes at infinity. Since $c(x) = u(x,1)/e_2(x,1) > 0$ on $\Gamma_M^+$, the condition $f(x_0, 1) \ge 0$ ensures $g(x_0, 1) = f(x_0, 1)/c(x_0) \ge 0$.
If $f(x_0, 1) = 0$, then $f \le 0$ everywhere, implying $g(x, 1) \le 0$ on the boundary. By the maximum principle, $g \le 0$ globally, which makes $(x_0, 1)$ a global maximum for $g$ (since $g(x_0, 1) = 0$). In this case, the Hopf boundary point lemma immediately yields $g_y(x_0, 1) \ge 0$. On the other hand, if $f(x_0, 1) > 0$, then $g(x_0, 1) > 0$. By the strong maximum principle, $g$ must attain its positive maximum on the boundary $\Gamma_M^+$. Since $c(x) \to 1$ and $c'(x) \to 0$ as $x \to \infty$, the asymptotic behavior of $g(x,1) = f(x,1)/c(x)$ coincides with that of $f(x,1)$. Consequently, there exists a constant $M_2 > 0$ such that for any $M \ge M_2$, $g(x,1)$ attains a local maximum at a point $x_*$ with $|x_* - x_0| \to 0$ as $M \to \infty$. At this boundary maximum $x_*$, the Hopf lemma forces that $g_y(x_*, 1) > 0$. By the uniform continuity of the gradient $\nabla g$, it follows that $g_y(x_0, 1) \ge 0$. This establishes the claim.
	
Since $f(x_0, 1) \geq 0$, taking $M \ge \max\{M_1, M_2\}$ sufficiently large ensures that the left-hand side of \eqref{e510} is strictly positive at point $(x_0,1)$, yielding a contradiction.
\end{proof}

With these intermediate results established, we proceed to complete the proof.
\begin{proof}[Proof of Lemma \ref{006lemma4.3}]
Fix $\eta^*, \zeta^*, \vartheta^*$ and $\alpha^*$ as in the statement, and recall that the associated function $w^*$ is defined in \eqref{w equations 1}. Let $\alpha_0 \in (\alpha^*, 1 - \gamma+\epsilon_1)$, and choose $M > 0$ such that $\|w^*\|_{C^1(\mathcal{R}_M^+)} < \tfrac{1}{2} \delta$, where $\delta = \delta(\alpha_0) > 0$ is the constant provided by Lemma \ref{006lemma4.6}.
Next, select $\varepsilon_M > 0$ such that Lemma \ref{006lemma4.5} holds with $w = w^*$, and define $
\varepsilon := \min\left( \varepsilon_M, \tfrac{1}{2} \delta, |\alpha_0 - \alpha^*| \right)$. Then, by Lemma \ref{006lemma4.5}, \eqref{0064.6} is satisfied. Specifically, $v \leq 0$ on $L_M^+$. Since $\|w\|_{C^1(\mathcal{R}_M^+)} < \delta$, Lemma \ref{006lemma4.6} further implies that $v < 0$ in $\mathcal{R}_M^+ \cup \Gamma_M^+$. The desired result now follows by combining \eqref{0064.6}.
\end{proof}

\begin{lemma}\label{006lemma7.5}
The nodal property \eqref{0064.1} holds along the global bifurcation curve $\mathscr{C}$.
\end{lemma}
\begin{proof}
By Theorem \ref{006theorem7.1}, the nodal property \eqref{0064.1} holds along the local curve \(\mathscr{C}_{\mathrm{loc}}\). We begin by proving that the global curve \(\mathscr{C}\) does not contain any trivial solutions \((0, \alpha)\). By the definition of \(\mathcal{U}\), all solutions along \(\mathscr{C}\) satisfy \(\alpha < \alpha_{\mathrm{cr}}\). Let \(\mathscr{T}\) denote the relatively closed set of all trivial solutions in \(\mathscr{C}\). Since \(\mathscr{C}\) is closed, \(\mathscr{T}\) is relatively closed in \(\mathscr{C}\). Moreover, Lemma \ref{006lemma5.6} shows that the linearized operator \(\mathscr{F}_w(0, \alpha)\) is invertible for all \(\alpha < \alpha_{\mathrm{cr}}\). Therefore, by the implicit function theorem, trivial solutions form locally unique continuous curves parametrized by \(w = w(\alpha)\), with \(\alpha < \alpha_{\mathrm{cr}}\). It follows that \(\mathscr{T}\) is relatively open in \(\mathscr{C}\). Since \(\mathscr{C}\) is connected, this implies that either (i) \(\mathscr{C}\) consists entirely of trivial solutions, or (ii) \(\mathscr{C}\) contains no trivial solutions. The first case is ruled out by the inclusion \(\mathscr{C}_{\mathrm{loc}} \subset \mathscr{C}\), as \(\mathscr{C}_{\mathrm{loc}}\) contains nontrivial solutions. Hence, \(\mathscr{C}\) contains no trivial solutions.\par
Next, define \(\mathscr{N} \subset \mathscr{C}\) as the subset of all \((w, \alpha) \in \mathscr{C}\) satisfying the nodal property \eqref{0064.1}. Since \(\mathscr{C}\) is connected in \(\mathcal{X} \times \mathbb{R}\)\ and contains no trivial solutions, Lemmas \ref{006lemma4.2} and \ref{006lemma4.3} imply that \(\mathscr{N}\) is both relatively open and closed in \(\mathscr{C}\). Because \(\mathscr{C}_{\mathrm{loc}} \subset \mathscr{N}\), it follows that \(\mathscr{N} \neq \varnothing\). By connectedness, we conclude that \(\mathscr{N} = \mathscr{C}\), i.e., the nodal property \eqref{0064.1} holds along the entire global curve.
\end{proof}
\subsection{Nonexistence of monotone bores}\label{025subsection5.2}
In this subsection, we establish the nonexistence of monotone bores, as stated in Theorem \ref{006theorem3.4}. At the same time, we lay the groundwork for excluding loss of compactness. The proof relies on key properties of the Bernoulli constant \eqref{0063.14} and the flow force \eqref{0063.15}, which are derived in Lemmas \ref{006lemma3.2} and \ref{006lemma3.3}.\par

Traditionally, the flow force defined by
$$
\int\left(P-P_{\text{atm}}+U^2\right)\, dY,
$$
plays an important role in the analysis of steady waves, where $P_{\text{atm}}$ is atmospheric pressure at the surface. Since we are studying EHD waves with overhanging profiles, the present definition of the flow force is motivated by the divergence form of the momentum equations in the case $\epsilon_{1}=0$. Accordingly, we define the flow force as
\begin{align}
S=\int_{x=\mathrm{constant}}
\begin{pmatrix}P-P_{\text{atm}}+U^2+\epsilon_1{E_2}^2
\end{pmatrix}dY-(UV-\epsilon_1{E_1}{E_2})dX.\notag
\end{align}
\par

In terms of the dimensionless variables,  \eqref{modified Bernoulli law} takes the form
$$
P+\frac{1}{2}(u^2+v^2)+\frac{\epsilon_1}{2}({e_1}^2+{e_2}^2)+\alpha (\eta-1)=\gamma\psi+\frac{1}{2}+\frac{\epsilon_1}{2}+P_{\text{atm}}-\gamma(1-\frac{1}{2}\gamma),
$$
which is equivalent to
\begin{align}
 P-P_{\text{atm}}+u^2+\epsilon_1{e_2}^2=\frac{1}{2}(u^2-v^2)+\frac{\epsilon_1}{2}({e_2}^2-{e_1}^2)-\alpha (\eta-1)+\gamma\psi+\frac{1}{2}+\frac{\epsilon_1}{2}-\gamma(1-\frac{1}{2}\gamma).\notag
\end{align}
Then, we have
\begin{equation}
\begin{split}
S=  \int_{0}^{1} \bigg(&\big(\frac{1}{2}(u^2-v^2)+\frac{\epsilon_1}{2}({e_2}^2-{e_1}^2)-\alpha (\eta-1)+\gamma(\zeta+\frac{1}{2} \gamma \eta^2-1+\frac{1}{2}\gamma)\\&\quad+\frac{1}{2}+\frac{\epsilon_1}{2}\big)\eta_y+ (uv-\epsilon_1{e_1}{e_2}) \eta_x \bigg)d y.
\end{split} \label{0063.4-}
\end{equation}
A careful calculation yields
\begin{equation}
\begin{split}
&\frac{1}{2}(u^2-v^2)\eta_y+\frac{\epsilon_1}{2}({e_2}^2-{e_1}^2)\eta_y+u v\eta_x-\epsilon_1{e_1}{e_2}\eta_x\\&=\frac{1}{2} \frac{\eta_{y}\left(\zeta_{y}^{2}-\zeta_{x}^{2}\right)+2 \eta_{x} \zeta_{x} \zeta_{y}}{\eta_{x}^{2}+\eta_{y}^{2}} +\frac{\epsilon_1}{2} \frac{\eta_{y}}{\eta_{x}^{2}+\eta_{y}^{2}}+\gamma\eta\zeta_y+\frac{1}{2}\gamma^2\eta^2\eta_y.
\end{split}\label{0063.4--}
\end{equation}
From \eqref{0063.4-} and \eqref{0063.4--}, we obtain
\begin{equation}
\begin{split}
S(x ; \eta, \zeta, \alpha):= & \frac{1}{2} \int_{0}^{1} \bigg(\frac{\eta_{y}\left(\zeta_{y}^{2}-\zeta_{x}^{2}\right)+2 \eta_{x} \zeta_{x} \zeta_{y}}{\eta_{x}^{2}+\eta_{y}^{2}} \bigg)\, d y +\frac{\epsilon_1}{2} \int_{0}^{1} \bigg( \frac{\eta_{y}}{\eta_{x}^{2}+\eta_{y}^{2}} \bigg)\, d y \\&-\left.\left(\frac{\gamma^{2}}{6} \eta^{3}+\frac{\alpha}{2} \eta^{2}-\frac{2\alpha+1+\epsilon_1}{2} \eta\right)\right|_{y=1}.
\end{split}\label{0063.4}
\end{equation}\par
To demonstrate that $S$ is independent of $x$, we first observe that the integrand in \eqref{0063.4} is the real part of holomorphic functions
\begin{equation}
\begin{split}
&\frac{{(\zeta_y+i\zeta_x)}^2}{\eta_y+i\eta_x}=\frac{\eta_y(\zeta_y^2-\zeta_x^2)+2\eta_x\zeta_x\zeta_y}{\eta_x^2+\eta_y^2}
+i\frac{2\eta_y\zeta_x\zeta_y-\eta_x({\zeta_y}^2-{\zeta_x}^2)}{\eta_x^2+\eta_y^2},\\
&\frac{1}{\eta_y+i\eta_x}=\frac{\eta_y}{\eta_x^2+\eta_y^2}
+i\frac{-\eta_x}{\eta_x^2+\eta_y^2}.
\end{split}\label{0063.0004}
\end{equation}
Then, by differentiating under the integral and using the Cauchy-Riemann equations, we conclude that
\begin{equation}
\begin{split}
&\frac{d}{dx}\int_{0}^{1}\frac{1}{2}\frac{\eta_{y}(\zeta_{y}^{2}-\zeta_{x}^{2})
+2\eta_{x}\zeta_{x}\zeta_{y}}{\eta_{x}^{2}+\eta_{y}^{2}}dy=\frac{1}{2}\frac{2\eta_{y}\zeta_{x}\zeta_{y}
-\eta_{x}(\zeta_{y}^{2}-\zeta_{x}^{2})}{\eta_{x}^{2}+\eta_{y}^{2}}\bigg|_{y=1},
\\&\frac{d}{dx}\int_{0}^{1}\frac{1}{2}\frac{\eta_{y}}{\eta_{x}^{2}+\eta_{y}^{2}}dy=
\frac{1}{2}\frac{-\eta_{x}}{\eta_{x}^{2}+\eta_{y}^{2}}\bigg|_{y=1}.
\end{split}\notag
\end{equation}
From \eqref{Non-dimensionalization equation d}, \eqref{Non-dimensionalization equation e} and \eqref{Non-dimensionalization equation f add}, we have
\begin{equation}
\begin{split}
\frac{1}{2}\frac{2\eta_{y}\zeta_{x}\zeta_{y}
-\eta_{x}(\zeta_{y}^{2}-\zeta_{x}^{2})}{\eta_{x}^{2}+\eta_{y}^{2}}
+\frac{\epsilon_1}{2}\frac{-\eta_{x}}{\eta_{x}^{2}+\eta_{y}^{2}}= \bigg(\frac{\gamma^{2}}{2} \eta^{2}+ \alpha \eta-\frac{2\alpha +1+\epsilon_1}{2}\bigg)\eta_x.
\end{split}\notag
\end{equation}
Substituting this into \eqref{0063.4} yields that $S$ is independent of $x$.
\begin{lemma}\label{006lemma3.2}
Let \(\hat{Q}\) be the function defined in \eqref{0063.14}, and suppose the asymptotic depth satisfies \(d > 0\). Then \(\hat{Q}\) is strictly convex and attains its unique minimum at \(d = d_{\mathrm{cr}}\). Moreover, there exists a unique value \(d_{*}\) such that \(\hat{Q}(d_{*}) = \hat{Q}(1)\). If \(\alpha < \alpha_{\mathrm{cr}}\), then \(d_{*} \in (d_{\mathrm{cr}}, \infty)\), whereas for \(\alpha > \alpha_{\mathrm{cr}}\), we have \(d_{*} \in (0, d_{\mathrm{cr}})\).
\end{lemma}
\begin{proof}
We begin by differentiating equation \eqref{0063.14} twice with respect to \(d\). The first derivative evaluated at \(d = 1\) is
\[
\hat{Q}'(1) = 2\big(\alpha - \alpha_{\mathrm{cr}}\big).
\]
The second derivative takes the form
\begin{equation}
\hat{Q}''(d) = \frac{3(2 - \gamma)^2}{2d^4} + \frac{\gamma^2}{2} + \frac{6\epsilon_1}{d^4} > 0. \label{0063.4++}
\end{equation}
This confirms that \(\hat{Q}\) is strictly convex for all \(d > 0\). Moreover, since \(\hat{Q}(d) \to \infty\) as \(d \to 0\) or \(d \to \infty\), the function must attain a unique minimum at some finite value \(d = d_{\mathrm{cr}}\).
\end{proof}
\begin{lemma}\label{006lemma3.3}
Let \(\hat{S}(d)\) be defined by \eqref{0063.15}. Then its derivative with respect to \(d\) is given by
\begin{align}
\hat{S}'(d) = \frac{1}{2} \left( \hat{Q}(1) - \hat{Q}(d) \right). \label{0063.17}
\end{align}
In particular, by the strict convexity of \(\hat{Q}\), it follows that \(\hat{S}(d_{*}) > \hat{S}(1)\) whenever \(\alpha < \alpha_{\mathrm{cr}}\), and \(\hat{S}(d_{*}) < \hat{S}(1)\) if \(\alpha > \alpha_{\mathrm{cr}}\).
\end{lemma}
\begin{proof}
By \eqref{0063.12} and \eqref{0063.4}, we compute
\[
\hat{S}(d) = \frac{(2 - \gamma)^2}{8d} - \frac{\gamma^2 d^3}{24} - \frac{(2 - \gamma)\gamma d}{4} - \frac{\alpha}{2} d^2 + \frac{2\alpha + 1+\epsilon_1}{2} d + \frac{\epsilon_1}{2d}.
\]
Differentiating with respect to \(d\) and applying the formula for \(\hat{Q}(d)\) from \eqref{0063.14}, we obtain \eqref{0063.17}.\par
Now suppose that \(\alpha < \alpha_{\mathrm{cr}}\). By Lemma~\ref{006lemma3.2}, \(\hat{Q}(d) < \hat{Q}(1)\) for all \(d \in (1, d_*)\), which implies
\[
\hat{S}(d_*) - \hat{S}(1) = \frac{1}{2} \int_{1}^{d_*} \left( \hat{Q}(1) - \hat{Q}(s) \right) ds > 0.
\]
\par
Conversely, if \(\alpha > \alpha_{\mathrm{cr}}\), then \(\hat{Q}(d) < \hat{Q}(1)\) holds for all \(d \in (d_*, 1)\), and we have
\[
\hat{S}(d_*) - \hat{S}(1) = \frac{1}{2} \int_{1}^{d_*} \left( \hat{Q}(1) - \hat{Q}(s) \right) ds = -\frac{1}{2} \int_{d_*}^{1} \left( \hat{Q}(1) - \hat{Q}(s) \right) ds < 0.
\]
This completes the proof.
\end{proof}
Since the domain under consideration is unbounded, the standard compact embeddings between H\"{o}lder spaces do not apply. However, for monotone waves, the only potential obstruction to compactness is the presence of a bore. The purpose of the following argument is to rule out this possibility.
\begin{proof}[\textbf{Proof of Theorem~\ref{006theorem3.4}}]
Suppose that \((\eta, \zeta, \vartheta)\) is a bore solution to \eqref{Non-dimensionalization equation a}-\eqref{Non-dimensionalization equation 2}. Then \eqref{0063.16} holds. Our goal is to show that \(\eta_{+} = \eta_{-}\), which is equivalent to proving
\[
\begin{aligned}
d_{+}\, y &= \hat{\eta}_{\mathrm{tr}}(y; d_{+}) \\
          &= \hat{\eta}_{\mathrm{tr}}(y; d_{-}) \\
          &= d_{-}\, y,
\end{aligned}
\]
that is, \(d_{+} = d_{-}\).\par
We first consider \(\alpha = \alpha_{\mathrm{cr}}\). Then Lemmas \ref{006lemma3.2} and \ref{006lemma3.3} immediately yield \(d_{+} = d_{-} = d_{*} = 1\).\par
Now assume \(\alpha \neq \alpha_{\mathrm{cr}}\). According to \eqref{0063.13} and Lemma \ref{006lemma3.2}, we know that
\[
d_{\pm} \in \{1, d_{*}\}.
\]
We claim that
\begin{align}
d_{-} = d_{*} \quad \text{or} \quad d_{+} = d_{*}. \label{0063.18+-}
\end{align}
Indeed, if \(d_{-} = 1\), then necessarily \(d_{+} = d_{*}\). Combining \eqref{0063.13}, \eqref{0063.16}, and \eqref{0063.18+-}, we find
\[
\hat{Q}(d_{*}) = \hat{Q}(1), \quad \hat{S}(d_{*}) = \hat{S}(1),
\]
which contradicts Lemma \ref{006lemma3.3}. Therefore, the only possibility is \(d_{-} = d_{+} = 1\), completing the proof.
\end{proof}
\subsection{Compactness and uniform regularity}\label{025subsection5.3}
The following lemma, together with Lemma \ref{006lemma7.5}, enables us to rule out alternative (ii) in Theorem \ref{006theorem7.4}.
\begin{lemma}[Compactness]\label{006Lemma 3.5}
Let
\((\eta_n, \zeta_n, \vartheta_n, \alpha_n)\) be a sequence of solutions to \eqref{Non-dimensionalization equations 1} satisfying
\begin{align}
\sup_{n} \|(\eta_n, \zeta_n, \vartheta_n, \alpha_n)\|_{C^{3+\beta}(\mathcal{R}) \times \mathbb{R}} < \infty, \quad
\inf_{n} \inf_{\mathcal{R}} \left(1 + \epsilon_1 - 2\alpha_n(\eta_n - 1)\right)^2 |\nabla \eta_n|^2> 0, \label{0063.18}
\end{align}
along with the monotonicity condition
\begin{align}
\partial_x \eta_n \leq 0 \quad \text{for } x \geq 0. \label{0063.19}
\end{align}
Then there exists a subsequence-still denoted by \((\eta_n, \zeta_n, \vartheta_n)\)-that converges to a limit \((\eta, \zeta, \vartheta)\) in \(C_{\mathrm{b}}^{3+\beta}(\overline{\mathcal{R}})\).
\end{lemma}
\begin{proof}
Without loss of generality, we assume that \(\alpha_n \to \alpha \in \mathbb{R}\). Our first claim is the asymptotic convergence
\begin{align}
\lim_{x \to \infty} \sup_n \sup_y \left|(\eta_n, \zeta_n, \vartheta_n)(x, y) - (y, (1-\gamma)y, y)\right| = 0. \label{0063.20}
\end{align}
We postpone the proof of this claim and proceed under its assumption.
\par
By Arzel\`{a}-Ascoli theorem and the uniform \(C^{3+\beta}(\mathcal{R})\) bounds, we extract a subsequence such that
\[
(\eta_n, \zeta_n, \vartheta_n) \to (\eta, \zeta, \vartheta) \quad \text{in } C_{\mathrm{loc}}^3(\overline{\mathcal{R}}) \text{ and } L^\infty(\overline{\mathcal{R}}),
\]
where \((\eta, \zeta, \vartheta)\) solves \eqref{Non-dimensionalization equations 1}. As a result, the differences
\[
v_n^{(1)} := \eta_n - \eta, \quad v_n^{(2)} := \zeta_n - \zeta, \quad v_n^{(3)} := \vartheta_n - \vartheta=0
\]
satisfy
\begin{align}
\|(v_n^{(1)}, v_n^{(2)}, v_n^{(3)})\|_{L^\infty(\mathcal{R})} \to 0 \quad \text{as } n \to \infty. \label{0063.21}
\end{align}
\par
To prove convergence in \(C_{\mathrm{b}}^{3+\beta}(\overline{\mathcal{R}})\), we examine the behavior on the surface \(\Gamma\), where the differences \(v_n^{(1)}, v_n^{(2)}\) satisfy the linearized boundary system
\begin{equation}
\begin{aligned}
a_{11} \partial_x v_n^{(1)} + a_{12} \partial_y v_n^{(1)} + a_{22} \partial_y v_n^{(2)} + b_1 v_n^{(1)} &= f_n, \\
c_1 v_n^{(1)} + c_2 v_n^{(2)} &= 0,
\end{aligned}
\end{equation}
with the coefficients defined by
\begin{align*}
a_{11} &= \left(2\alpha_n(\eta_n - 1) + 2\alpha(\eta - 1) - 2 - 2\epsilon_1\right) \eta_x, \\
a_{12} &= \left(2\alpha_n(\eta_n - 1) + 2\alpha(\eta - 1) - 2 - 2\epsilon_1\right) \eta_y + a_{22} c_1, \\
a_{22} &= \zeta_{ny} + \gamma \eta_n \eta_{ny} + \zeta_y + \gamma \eta \eta_y,  \\
b_1 &= \gamma (\eta_{ny} + \eta_y + \eta + \eta_n) a_{22}, \quad c_1 = \gamma (\eta + \eta_n), \quad c_2 = 1.
\end{align*}
From the uniform bounds in \eqref{0063.18}, all coefficients and \(f_n\) are uniformly bounded in \(C_{\mathrm{b}}^{2+\beta}(\Gamma)\). Moreover, we have that
\[
\left(c_1a_{21}-c_2a_{11}\right)^2+\left(c_1a_{22}-c_2a_{12}\right)^2=\left(2\alpha_n(\eta_n - 1) + 2\alpha(\eta - 1) - 2 - 2\epsilon_1\right)^2 (\eta_x^2 + \eta_y^2) \geq \delta \quad \text{on } \Gamma
\]
for some fixed \(\delta > 0\), where $a_{21}=0$ and the last inequality is based on \eqref{Regional boundary condition} and \eqref{0063.18}.
Thus, applying \cite[Theorem A.1]{SVHMHW2023}, together with the asymptotic condition \eqref{w equation 2} and the convergence \eqref{0063.21}, we obtain
\[
\|(v_n^{(1)}, v_n^{(2)}, v_n^{(3)})\|_{C^{3+\beta}(\mathcal{R})} \leq C\left( \|f_n\|_{C^{2+\beta}(\Gamma)} + \|(v_n^{(1)}, v_n^{(2)}, v_n^{(3)})\|_{L^\infty(\mathcal{R})} \right) \to 0,
\]
as \(n \to \infty\). Hence,
\[
(\eta_n, \zeta_n, \vartheta_n) \to (\eta, \zeta, \vartheta) \quad \text{in } C_{\mathrm{b}}^{3+\beta}(\overline{\mathcal{R}}).
\]
\par
We now prove the claim \eqref{0063.20} by contradiction. Suppose that \eqref{0063.20} fails. Then there exist a sequence \(\{(x_n, y_n)\} \subset \mathbb{R}^2\) with \(x_n \to \infty\) and a constant \(\varepsilon > 0\) such that
\[
\left|(\eta_n, \zeta_n, \vartheta_n)(x_n, y_n) - (y_n, (1 - \gamma)y_n, y_n)\right| \geq \varepsilon
\]
for all \(n\). Using a translation argument and the monotonicity condition \eqref{0063.19}, we extract a limiting profile that yields a bore-type solution to equations \eqref{Non-dimensionalization equation a}-\eqref{Non-dimensionalization equation 2} as \(n \to \infty\). This contradicts Theorem \ref{006theorem3.4}, which excludes the existence of such bore solutions. Therefore, \eqref{0063.20} must hold.
\end{proof}
We assume that alternative (i) does not hold. If alternative (i) does not hold, then alternative (ii) in Theorem \ref{006theorem7.4} must be valid.
\begin{rmk}
The fact that alternative (i) in Theorem \ref{006theorem7.4} does not occur guarantees the validity of the first condition in \eqref{0063.18}.
\end{rmk}
\begin{lemma}\label{006lemma7.6}
Alternative (ii) in Theorem \ref{006theorem7.4} cannot occur.
\end{lemma}
\begin{proof}
The conclusion follows directly from Lemmas \ref{006lemma7.5} and \ref{006Lemma 3.5}.
\end{proof}
This leads to a contradiction. Therefore, alternative (i) must hold. To support the analysis of the terms in Theorem \ref{006theorem7.4} (i) presented in subsection \ref{025subsection5.5}, we begin by introducing two auxiliary propositions. The first proposition establishes that the $C^{3+\beta}$-norms of the functions $\eta$, $\zeta$ and $\vartheta$ are uniformly bounded in terms of a positive constant $\delta$, which appears in the given inequalities
\begin{align}
\delta \leq |\nabla \eta| \leq \frac{1}{\delta} \quad \text{and} \quad 1 +\epsilon_1- 2\alpha(\eta - 1) \geq \delta \quad \text{in } \mathcal{R}.\label{0066.1}
\end{align}
\begin{proposition}[Uniform regularity]\label{006proposition6.1}
Assume that \((\eta, \zeta, \vartheta, \alpha)\) is a solution to \eqref{Non-dimensionalization equations 1} with \(0 \leq \alpha \leq \alpha_{\mathrm{cr}}\), and that \eqref{0066.1} holds for some \(\delta > 0\). Then, there exists a constant \(C = C(\delta) > 0\) such that $\|\eta\|_{C^{3+\beta}(\mathcal{R})} < C, \|\zeta\|_{C^{3+\beta}(\mathcal{R})} < C$ and $\|\vartheta\|_{C^{3+\beta}(\mathcal{R})}<C$.
\end{proposition}
The proof proceeds in a manner similar to Proposition 6.1 in \cite{SVHMHW2023}, with the principal difference being that the system \eqref{Non-dimensionalization equations 1} is decomposed into two coupled scalar equations, as described in detail below.\par
Recall from \eqref{New variable representation},
\begin{align}
\Psi(x, y) =\zeta(x, y) + \frac{1}{2} w \eta^2(x, y),\quad \theta(x, y)=\vartheta(x, y).\label{0066.2}
\end{align}
Fixing $\eta$, we consider the equations governing $\Psi$, derived from \eqref{Non-dimensionalization equation b}, \eqref{Non-dimensionalization equation d}, and \eqref{Non-dimensionalization equation h}. This leads to the following boundary value problem
\begin{equation}
\begin{split}
\Delta \Psi = \gamma |\nabla \eta|^2&\quad\text{in } \mathcal{R}, \\
\Psi = 1-\frac{1}{2}\gamma&\quad\text{on } \Gamma, \\
\Psi = 0 &\quad\text{on } \mathcal{B}.
\end{split}\label{0066.3}
\end{equation}
Finally, with $\Psi$ and $\theta$ fixed, the equation for $\eta$ follows from \eqref{Non-dimensionalization equation a}, \eqref{Non-dimensionalization equation e}, and \eqref{Non-dimensionalization equation g}, yielding:
\begin{subequations}\label{0066.4}
\begin{align}
\Delta \eta = 0&\quad\text{in } \mathcal{R},\label{0066.4a} \\
(1+\epsilon_1 - 2\alpha (\eta - 1)) |\nabla \eta|^2 = \Psi_y^2+\epsilon_1&\quad\text{on } \Gamma, \label{0066.4b}\\
\eta = 0&\quad\text{on } \mathcal{B}.\label{0066.4c}
\end{align}
\end{subequations}
\begin{proposition}\label{006proposition6.5}
Assume that \(\Psi\) as defined in \eqref{0066.2}, solve the boundary value problem \eqref{0066.3}, and that \((\eta, \zeta, \vartheta, \alpha)\) is a solution to \eqref{Non-dimensionalization equations 1} with \(0 \leq \alpha \leq \alpha_{\mathrm{cr}}\). Then the following bounds hold: \\
(i) if $\gamma\leq 0$ then $\Psi _{y}< 1- \frac{1}{2}\gamma$ on $\Gamma$,  \\
(ii) if $\gamma \geq 0$ then $\Psi _{y}>\min\left\{2-\gamma+2\epsilon_1, \gamma\inf_{\mathcal{R}}|\nabla\eta|^{2}\right\}$ on $\Gamma$.
\end{proposition}
\begin{proof}
(i) Suppose \(\gamma \leq 0\). Define
$$\tilde{\Psi}=\Psi-\bigl(1-\frac{1}{2}\gamma)y.$$
Then \(\tilde{\Psi}\) solves
\begin{eqnarray}
\left\{\begin{array}{llll}{\Delta\tilde{\Psi}=\gamma |\nabla\eta|^2}&{\quad\mathrm{in~}\mathcal{R}},\notag\\
{\tilde{\Psi}=0}&{\quad\mathrm{on~}\Gamma},\notag\\
{\tilde{\Psi}=0}&{\quad\mathrm{on~}\mathcal{B}}.\end{array}\right. \notag
\end{eqnarray}
By the strong minimum principle, $\Psi$ must attain its minimum on the boundary. In particular, by applying the Hopf boundary point lemma at any boundary point, we obtain that
$$\tilde{\Psi}_y=\Psi_y-1+\frac{1}{2}\gamma<0.$$
This implies (i).\par
(ii) Now suppose \(\gamma \geq 0\). Define
$$\bar{\Psi}=\Psi-My^2,$$
where
$$M=\min\left\{1-\frac{1}{2}\gamma+\epsilon_1, \frac{1}{2}\gamma\inf_{\mathcal{R}}|\nabla\eta|^{2}\right\}.$$
Since \(0 \leq \alpha \leq \alpha_{\mathrm{cr}}\), we have \(1 - \gamma + \epsilon_1 > 0\), which ensures that \(M > 0\). Then \(\bar{\Psi}\) satisfies
\begin{eqnarray}
\left\{\begin{array}{llll}{\Delta\bar{\Psi}=\gamma|\nabla \eta|^2-M}&{\quad\mathrm{in~}\mathcal{R}},\notag\\
{\bar{\Psi}=1-\frac{1}{2}\gamma-M}&{\quad\mathrm{on~}\Gamma},\notag\\
{\bar{\Psi}=0}&{\quad\mathrm{on~}\mathcal{B}}.
\end{array}\right. \notag
\end{eqnarray}
By the strong maximum principle, \(\bar{\Psi}\) attains its maximum at any boundary point. At such a point, applying the Hopf boundary point lemma, we get that
$$\bar{\Psi}_y=\Psi_y-2M>0,$$
which establishes (ii).
\end{proof}

\subsection{Bounds on the Froude number}\label{025subsection5.4}
This subsection is devoted to the proof of Theorem \ref{006theorem3.8}, which establishes a lower bound for the Froude number. It also facilitates the analysis of the terms in Theorem \ref{006theorem7.4} (i). In contrast to the approach taken in \cite{MHW20152}, the presence of internal stagnation points and the possibility of overhanging profiles prevent a direct application of their method. Therefore, we adopt the strategy developed in \cite{SVHMHW2023}.\par
We begin by introducing a modified form of the fluid force flux function
\begin{equation}
\begin{split}
\Phi(x, y) := \int_0^y \bigg( \frac{\eta_y(\zeta_y^2 - \zeta_x^2) + 2\eta_x \zeta_x \zeta_y}{\eta_x^2 + \eta_y^2}+\epsilon_1  \frac{\eta_{y}}{\eta_{x}^{2}+\eta_{y}^{2}}
+ (1 - \gamma^2)\eta_y + 2(\gamma - 1-\frac{\epsilon_1}{2}) \bigg) dy.
\end{split}\label{0063.23}
\end{equation}
As $|x| \to \infty$, it follows from \eqref{w equation 2} that $\Phi \to 0$. Observe that the first two terms in \eqref{0063.23} coincide with the integrand of the flow force expression given in \eqref{0063.4}, with the distinction that the upper limit of integration here is the variable $y$, rather than the fixed value $y = 1$.\par
To proceed, we analyze the asymptotic behavior of \eqref{0063.4} as $|x| \to \infty$:
\begin{equation}\notag
\begin{split}
S(x ; \eta, \zeta, \alpha)\big|_{|x| \to \infty}
= & \frac{1}{2} \int_0^1 \big((1 - \gamma)^2+\epsilon_1\big) \, dy - \left. \left( \frac{\gamma^2}{6} \eta^3 + \frac{\alpha}{2} \eta^2 - \frac{2\alpha + 1+\epsilon_1}{2} \eta \right) \right|_{y=1} \\
= & \frac{\gamma^2}{3} - \gamma + \frac{\alpha}{2} + 1+\epsilon_1.
\end{split}
\end{equation}
Since $S$ does not depend on $x$, it follows that
\begin{equation}\notag
\begin{split}
0 = & \frac{1}{2} \int_0^1 \frac{\eta_y (\zeta_y^2 - \zeta_x^2) + 2 \eta_x \zeta_x \zeta_y}{\eta_x^2 + \eta_y^2} \, dy +\frac{\epsilon_1}{2} \int_{0}^{1} \bigg( \frac{\eta_{y}}{\eta_{x}^{2}+\eta_{y}^{2}} \bigg)\, d y \\
& - \left. \left( \frac{\gamma^2}{6} \eta^3 + \frac{\alpha}{2} \eta^2 - \frac{2\alpha + 1+\epsilon_1}{2} \eta \right) \right|_{y=1} - \left( \frac{\gamma^2}{3} - \gamma + \frac{\alpha}{2} + 1 +\epsilon_1\right),
\end{split}
\end{equation}
which simplifies to
\begin{equation}
\begin{split}
&\frac{1}{2} \int_0^1 \frac{\eta_y (\zeta_y^2 - \zeta_x^2) + 2 \eta_x \zeta_x \zeta_y}{\eta_x^2 + \eta_y^2} \, dy +\frac{\epsilon_1}{2} \int_{0}^{1} \bigg( \frac{\eta_{y}}{\eta_{x}^{2}+\eta_{y}^{2}} \bigg)\, d y
\\&= \frac{\gamma^2}{3} - \gamma + \frac{\alpha}{2} + 1+\epsilon_1 + \frac{\gamma^2}{6} \eta^3 + \frac{\alpha}{2} \eta^2 - \frac{2\alpha + 1+\epsilon_1}{2} \eta.
\end{split} \label{006flow+}
\end{equation}
Taking \eqref{006flow+} into \eqref{0063.23}, we have
\begin{equation}\label{0063.24b}
\Phi = \frac{\gamma^2}{3} (\eta - 1)^3 + (\alpha + \gamma^2)(\eta - 1)^2-\epsilon_1 \eta+\epsilon_1\quad \text{on } \Gamma.
\end{equation}
It follows directly from the definition of $\Phi$ in \eqref{0063.23} that
\begin{align}
\Phi = 0 \quad \text{on } \mathcal{B}. \notag
\end{align}
Note that the first two terms in \eqref{0063.23} represent the real part of the holomorphic functions given in \eqref{0063.0004}. Therefore, we conclude that
\begin{align}
\Delta \Phi = 0 \quad \text{in } \mathcal{R}. \label{0063.24a}
\end{align}\par
Multiplying \eqref{0063.24a} by $y$ and integrating by parts twice, for any $M > 0$, we have
\[
\begin{aligned}
0 = -\int_{-M}^{M} \int_{0}^{1} \Delta \Phi \cdot y \, dy \, dx = \int_{-M}^{M} \Phi \, dx \bigg|_{y=0}^{y=1} - \int_{-M}^{M} \Phi_y \cdot y \, dx \bigg|_{y=0}^{y=1} - \int_{0}^{1} \Phi_x \cdot y \, dy \bigg|_{x=-M}^{x=M},
\end{aligned}
\]
which leads
\begin{align}
\left. \int_{-M}^{M} (\Phi - \Phi_y\cdot y) \, dx \right|_{y=0}^{y=1} = \int_{0}^{1} \Phi_x \cdot y \, dy \bigg|_{x=-M}^{x=M} = o(1) \quad \text{as} \, M \to \infty.\label{0063.26}
\end{align}
By \eqref{w equations 1}, the final equality holds since $\Phi_x \to 0$ as $x \to \infty$. Next, we observe that the second term in the integral on the left-hand side of \eqref{0063.26}, $\Phi_y \cdot y$, vanishes along the boundary $\mathcal{B}$. Using \eqref{Non-dimensionalization equation d}, \eqref{Non-dimensionalization equation e}, and \eqref{Non-dimensionalization equation f add}, we obtain
\begin{equation}
\begin{split}
\frac{\eta_y(\zeta_y^2-\zeta_x^2)+2\eta_x\zeta_x\zeta_y}{\eta_x^2+\eta_y^2}
+\epsilon_1\frac{\eta_{y}}{\eta_{x}^{2}+\eta_{y}^{2}}=
\eta_y\big(1+\epsilon_1-2\alpha(\eta-1)\big)+\eta_y\gamma^2\eta^2.
\end{split}\label{0063.26+++}
\end{equation}
Then, by \eqref{0063.23} and \eqref{0063.26+++}, we get
\begin{equation}
\begin{split}
\Phi_y &= \frac{\eta_y(\zeta_y^2-\zeta_x^2)+2\eta_x\zeta_x\zeta_y}{\eta_x^2+\eta_y^2}+\epsilon_1\frac{\eta_{y}}{\eta_{x}^{2}+\eta_{y}^{2}}
+(1-\gamma^2)\eta_y+2(\gamma-1-\frac{\epsilon_1}{2}) \\
&=-2\gamma\eta\zeta_y+(2+\epsilon_1+2\alpha-\gamma^2)\eta_y-2\alpha\eta\eta_y-\gamma^2\eta^2\eta_y+2(\gamma-1-\frac{\epsilon_1}{2})
\end{split}\quad\text{ on }\Gamma. \label{0063.27}
\end{equation}
Since both $\eta_x$ and $\zeta_x$ tend to zero as $x \to \infty$, the Gauss-Green theorem, combined with the kinematic boundary condition \eqref{Non-dimensionalization equation d}, implies that
\begin{equation}
\begin{split}
\int_{-M}^{M} \eta \zeta_y \, dx &= \int_{-M}^{M} \eta_y \zeta \, dx + o(1) \\&= \int_{-M}^{M} \eta_y \big(1-\frac{1}{2}\gamma-\frac{1}{2} \gamma\eta^{2} \big)\, dx + o(1)
\end{split}\text{ as } M \to \infty. \label{0063.28}
\end{equation}
Combining \eqref{0063.27} and \eqref{0063.28}, and rewriting the variables in terms of $w_1 = \eta - y$, we obtain
\begin{align}
&\int_{-M}^{M} \Phi_y \cdot y \, dx \bigg|_{y=0}^{y=1}\notag\\&= \int_{-M}^{M} 2 \big( (1 - \gamma+\frac{\epsilon_1}{2}+\alpha)w_{1y}-\alpha w_1w_{1y}-\alpha w_1-\alpha w_{1y}\big) \, dx \bigg|_{y=1} + o(1) \quad \text{as} \, M \to \infty.\notag
\end{align}
Applying the Gauss-Green theorem to \(w_1\) and \(y\), we obtain
\[
\int_{-M}^{M} w_{1y} \, dx = \int_{-M}^{M} w_1 \, dx + o(1) \quad \text{as} \, M \to \infty.
\]
Taking \eqref{0063.24b} as \(M \to \infty\), we have
\begin{align}
&\int_{-M}^{M} (\Phi - \Phi_y \cdot y) \, dx \bigg|_{y=0}^{y=1}\notag \\&=2\int_{-M}^{M} \bigg((\gamma - 1-\epsilon_1+ \alpha )w_1 +\frac{\gamma^2}{6} w_1^3+\frac{\alpha+\gamma^2}{2}w_1^2+\alpha w_1w_{1y}\bigg)\, dx\notag\\&\quad+ o(1).\notag
\end{align}
By combining \eqref{0063.26}, we obtain
\begin{equation}
\begin{split}
\left(1 - \gamma+ \epsilon_1 - \alpha \right) \int_{-M}^{M} w_1 \, dx &=
\alpha \int_{-M}^{M} w_1 w_{1y} \, dx + \frac{\alpha + \gamma^2}{2} \int_{-M}^{M} w_1^2 \, dx \\
&\quad + \frac{\gamma^2}{6} \int_{-M}^{M} w_1^3 \, dx + o(1)
\end{split} \label{0063.25}
\end{equation}
as \(M \to \infty\).
\begin{proof}[\textbf{Proof of Theorem~\ref{006theorem3.8}}]
Assume that \( w_1 \not\equiv 0 \). Then \( w_1 \) must be strictly positive at some point on \( \Gamma \). Our goal is to determine the sign of the term involving \( w_1 w_{1y} \).\par
Observe that
\begin{align}
0 &< \int_{-M}^{M} \int_{0}^{1} |\nabla w_1|^2 \, dy \, dx = \int_{-M}^{M} \int_{0}^{1} \nabla \cdot (w_1 \nabla w_1) \, dy \, dx \notag \\
&= \int_{-M}^{M} w_1 w_{1y} \, dx \bigg|_{y=1} + \int_{0}^{1} w_1 w_{1x} \, dy \bigg|_{x=-M}^{x=M}, \notag
\end{align}
where we have applied the divergence theorem to \(w_1 \nabla w_1\). Since \( w_1 w_{1x} \to 0 \) as \( |x| \to \infty \), for sufficiently large \( M \) the second term vanishes, and we conclude that $\int_{-M}^{M} w_1 w_{1y} \, dx > 0$.\par
All terms on the right-hand side of the integral identity \eqref{0063.25} are positive. Hence, for sufficiently large \( M \), the left-hand side must also be positive. This yields
\[
1 - \gamma + \epsilon_1 - \alpha > 0.
\]
We complete the proof.
\end{proof}
We now exclude the fourth term in alternative (i) of Theorem \ref{006theorem7.4}.
\begin{lemma}\label{006lemma7.7}
If both \( \|w(s)\|_\chi \) and \( 1/\lambda(w(s), \alpha(s)) \) remain uniformly bounded along the curve \( \mathscr{C} \), then
\[
\limsup_{s \to \infty} \alpha(s) < \alpha_{\mathrm{cr}}.
\]
\end{lemma}
\begin{proof}
Suppose, for contradiction, that there exists a sequence \( s_n \to \infty \) such that
\[
\sup_{n} \|w(s_n)\|_\chi < \infty, \quad \inf_{n} \lambda(w(s_n), \alpha(s_n)) > 0, \quad \text{and} \quad \alpha(s_n) \to \alpha_{\mathrm{cr}}.
\]
By Lemma \ref{006Lemma 3.5}, we extract a subsequence (still denoted \( s_n \)) such that
\[
(w(s_n), \alpha(s_n)) \to (w^*, \alpha^*) \in \mathcal{X} \times \mathbb{R},
\]
where \( (w^*, \alpha^*) \) is a solution to \( \mathscr{F}(w, \alpha) = 0 \) and \( \alpha^* = \alpha_{\mathrm{cr}} \). Since \( w_1 \geq 0 \) on the surface \( \Gamma \), Theorem \ref{006theorem3.8} implies that the limiting solution must be trivial, i.e., \( w= 0 \). Thus, we have \( \|w(s_n)\|_\chi \to 0 \). In addition, by Lemma \ref{006lemma7.5}, all \( w(s_n) \) satisfy the nodal property \eqref{0064.1}. Therefore, for sufficiently large \( n \), Theorem \ref{006theorem7.1} (ii) implies that \( (w(s_n), \alpha(s_n)) \in \mathscr{C}_{\mathrm{loc}} \). This contradicts the global alternative in Theorem \ref{006theorem7.4} (c), and the result follows.
\end{proof}
\subsection{Proof of Theorem \ref{006theorem7.9ABOVE}}\label{025subsection5.5}
Theorem \ref{006theorem7.9ABOVE} is now ready to be proved. As noted in Section \ref{025section2.4}, the primary focus lies on the case $\gamma < 0$, with the corresponding results for $\gamma \geq 0$ being stated only for completeness.
\begin{theorem} \label{006theorem7.9}
Fix the gravitational constant \(g > 0\), the asymptotic depth \(d > 0\), $\gamma$ and permittivity $\epsilon_1 > 0$. Then there exists a global continuous curve \(\mathscr{C}\) of solutions to \eqref{Non-dimensionalization equations 1}, parameterized by \(s \in (0, \infty)\). Moreover, the following asymptotic properties hold along \(\mathscr{C}\) as \(s \to \infty\):
	\begin{itemize}
		\item[(i)] For \(\gamma < 0\),
		\begin{align}
			\min\Biggl\{ & \inf_{\Gamma} \bigl(1+\epsilon_1 - \frac{2}{F^2} \frac{\eta - d}{d} \bigr),\;
			\inf_{\Gamma} |\nabla \eta(s)|,\;
			\inf_\Gamma \biggl(\bigl(1 + \epsilon_1 - \frac{2}{F^2} \frac{\eta - d}{d}\bigr) |\nabla \eta(s)|^2-\epsilon_1\biggr), \notag\\
			& \frac{1}{F(s)} \Biggr\}\longrightarrow 0,\notag
		\end{align}
		
		\item[(ii)] For \(\gamma > 0\),
		\begin{align}
			\min\Biggl\{
			& \inf_{\Gamma} |\nabla \eta(s)|,\; \big( \sup_{\Gamma} |\nabla \eta(s)| \big)^{-1},\;
			\inf_\Gamma \left(\big(1 + \epsilon_1 - \frac{2}{F^2} \frac{\eta - d}{d}\big) |\nabla \eta(s)|^2-\epsilon_1\right), \notag\\
			& \frac{1}{F(s)} \Biggr\} \longrightarrow 0,\notag
		\end{align}
		
		\item[(iii)] For \(\gamma = 0\),
		\begin{align}
			\min\Biggl\{
			& \inf_{\Gamma} \big(1+\epsilon_1 - \frac{2}{F^2} \frac{\eta - d}{d} \big),\;
			\inf_\Gamma \left(\big(1 + \epsilon_1 - \frac{2}{F^2} \frac{\eta - d}{d}\big) |\nabla \eta(s)|^2-\epsilon_1\right), \notag\\
			& \frac{1}{F(s)} \Biggr\} \longrightarrow 0,\notag
		\end{align}
	\end{itemize}
	where \(F\) denotes the Froude number. All solutions on \(\mathscr{C}\) are symmetric and monotone waves of elevation. Specifically, \(\eta\) is even in \(x\), and \(\eta_x(x, d) < 0\) for \(x > 0\).
\end{theorem}
Most of the essential elements have already been established in the preceding analysis, it remains to synthesize these findings to complete the argument.
\begin{proof}[\textbf{Proof of Theorem \ref{006theorem7.9}}]
Let \(\mathscr{C}\) denote the global curve established in Theorem \ref{006theorem7.4}. By Lemma \ref{006lemma7.7}, we conclude that alternative (i) in Theorem \ref{006theorem7.4} must occur:
\begin{align}
\|w(s)\|_\chi +\frac{1}{\kappa(w(s), \alpha(s))-\epsilon_1}+ \frac{1}{\lambda(w(s), \alpha(s))}+ \frac{1}{\alpha(s)} \longrightarrow \infty \quad \text{as } s \to \infty. \label{0067.6}
\end{align}
From the definition of \(\lambda(w, \alpha)\) in \eqref{0253.2}, and by applying the maximum principle and the maximum modulus principle to its two factors, we deduce
\[
\frac{1}{\lambda(w(s), \alpha(s))} \leq \frac{1}{\inf_\Gamma (1 + \epsilon_1 - 2\alpha(s) w_1(s))} + \frac{1}{\inf_\Gamma (w_{1x}^2(s) + (1 + w_{1y}(s))^2)}.
\]
According to Proposition \ref{006proposition6.1}, the norm \( \|w(s)\|_\chi \) is bounded in terms of the parameter \(\delta > 0\) from \eqref{0066.1}. Using the maximum principle again, we estimate
\[
\|w(s)\|_\chi \leq \frac{1}{\inf_\Gamma (1 + \epsilon_1 - 2\alpha(s) w_1(s))} + \sup_\Gamma |\nabla w_1(s)| + \frac{1}{\inf_\Gamma (w_{1x}^2(s) + (1 + w_{1y}(s))^2)}.
\]
Substituting into \eqref{0067.6}, we obtain
\begin{align}
&\frac{1}{\inf_\Gamma (1 + \epsilon_1 - 2\alpha(s) w_1(s))} + \sup_\Gamma |\nabla w_1(s)| + \frac{1}{\inf_\Gamma (w_{1x}^2(s) + (1 + w_{1y}(s))^2)} \notag \\
&\quad +\frac{1}{\inf_{\Gamma} (1 + \epsilon_1 - 2\alpha(s) w_1(s)) |\nabla w_1(s)|^2-\epsilon_1}+ \frac{1}{\alpha(s)} \longrightarrow \infty \quad \text{as } s \to \infty. \label{0067.7}
\end{align}
To simplify \eqref{0067.7}, we interpret its components using the dimensional variables (denoted with a superscript *). Recall that \(\alpha = \frac{1}{F^2}\) and
\[
1 + \epsilon_1 - 2\alpha \frac{\eta^* - d}{d} = 1 + \epsilon_1 - 2\alpha (\eta - 1) = 1 + \epsilon_1 - 2\alpha w_1,
\]
\[
|\nabla \eta^*(x^*, y^*)|^2 = |\nabla \eta(x, y)|^2 = w_{1x}^2(x, y) + (1 + w_{1y}(x, y))^2.
\]\par
(i) Case \(\gamma < 0\): from Proposition \ref{006proposition6.5} (i), we know
\begin{align}
\Psi_y < 1 - \frac{1}{2} \gamma, \quad \theta_y < 1\quad \text{on } \Gamma. \label{0067.8}
\end{align}
We claim that \( \Psi_y^2> 0 \) on \(\Gamma\). This will be verified shortly. Substituting into \eqref{0066.4b} yields
\[
(1 +\epsilon_1- 2\alpha w_1)(w_{1x}^2 + (1 + w_{1y})^2) = \Psi_y^2  +\epsilon_1\leq (1 - \frac{1}{2}\gamma)^2+\epsilon_1.
\]
This shows that the second term in \eqref{0067.7} is controlled by a multiple of the third. Hence,
\begin{align}
&\min\left\{ \inf_\Gamma\big(1 + \epsilon_1 - \frac{2}{F^2} \frac{\eta - d}{d} \big), \inf_\Gamma |\nabla \eta(s)|, \inf_\Gamma \left(\big(1 + \epsilon_1 - \frac{2}{F^2} \frac{\eta - d}{d}\big) |\nabla \eta(s)|^2-\epsilon_1\right),
\frac{1}{F(s)} \right\}\notag\\ &\longrightarrow 0.\notag
\end{align}
To prove the claim, note that from \eqref{0066.4b} and the definition \eqref{0253.2}, we have
$$\Psi_y^2+\epsilon_1 \geq \kappa(w, \alpha)>\epsilon_1 \text{ on } \Gamma,$$
which proves the assertion and completes the proof of Theorem~\ref{006theorem7.9} (i).\par
(ii) Case \(\gamma \geq 0\): combining \eqref{0066.4b} with Proposition \ref{006proposition6.5} (ii), we obtain
\[
1 + \epsilon_1 - 2\alpha w_1 \geq \frac{\left(\min\left\{2 - \gamma + 2\epsilon_1,\ \gamma \inf_{\mathcal{R}} (w_{1x}^2 + (1 + w_{1y})^2)\right\}\right)^2}{w_{1x}^2 + (1 + w_{1y})^2}.
\]
This shows that the first term in \eqref{0067.7} is bounded above by a constant multiple of the second and third terms, which completes the proof of Theorem \ref{006theorem7.9} (ii).\par
(iii) Case \(\gamma = 0\): Proposition \ref{006proposition6.5} (i) ensures
\begin{align}
\Psi_y < 1 \quad \text{on } \Gamma. \label{0067.8+1}
\end{align}
We claim again that \(\Psi_y^2 > 0\). This will be shown shortly. Substituting into \eqref{0066.4b} yields
\[
(1 + \epsilon_1 - 2\alpha w_1)(w_{1x}^2 + (1 + w_{1y})^2) = \Psi_y^2 + \epsilon_1 \leq 1 + \epsilon_1.
\]
Thus, the second term in \eqref{0067.7} is again controlled by the third. To establish the claim, note that as before, $\Psi_y^2+\epsilon_1 \geq \kappa(w, \alpha)>\epsilon_1$ on $\Gamma$, which verifies the claim. Consequently, this ensures that the subsequent two terms in \eqref{0067.7} are dominated by the leading term. This completes the proof of Theorem \ref{006theorem7.9} (iii).
\end{proof}
\begin{rmk}	\label{025rmk2.2}
As $\gamma>0$, if asymptotic property $\left(\sup_{\Gamma}|\nabla\eta(s)|\right)^{-1}\rightarrow 0$ holds, which means that the conformal transformation of variables degenerates and the free surface expands until it loses smoothness.
\end{rmk}
Since the proof of Theorem \ref{006theorem7.9ABOVE} follows closely the same reasoning as that of Theorem \ref{006theorem7.9}, it is omitted for brevity.、

\section*{Data Availability Statements}
Data sharing not applicable to this article as no datasets were generated or analysed during the current study.

\section*{Conflict of interest}
All authors have the same contribution to the article and enjoy equal status and all authors declare that they have no conflict of interest.


\begin{thebibliography}{s2}
\setlength{\itemsep}{0pt}
\bibitem{CJAJFT1981}C.J.~Amick, J.F.~Toland,
\newblock On periodic water-waves and their convergence to solitary waves in the long-wave limit. Philos. Trans. Roy. Soc. London Ser. A 303 (1981), no. 1481, 633--669.

\bibitem{CJAJFT19812}C.J.~ Amick, J.F.~Toland,
\newblock On solitary water-waves of finite amplitude. Arch. Rational Mech. Anal. 76 (1981), no. 1, 9--95.

\bibitem{SAADLN1964}S.~Agmon, A.~Douglis, L.~Nirenberg,
\newblock Estimates near the boundary for solutions of elliptic partial differential equations satisfying general boundary conditions. II. Commun. Pure Appl. Math. 17 (1964), 35--92.

\bibitem{JTB1977}J.T.~Beale,
\newblock The existence of solitary water waves. Comm. Pure Appl. Math. 30 (1977), no. 4, 373--389.

\bibitem{BBJT2003}B.~Buffoni, J.~Toland,
\newblock Analytic Theory of Global Bifurcation: An Introduction. Princeton Series in Applied Mathematics. Princeton University Press, Princeton (2003).

\bibitem{Constantin}A.~Constantin, E.~Varvaruca,
Steady periodic water waves with constant vorticity: regularity and local bifurcation. Arch. Ration. Mech. Anal. 199 (2011), no. 1, 33--67.

\bibitem{ACWS2004}A.~Constantin, W.~Strauss,
\newblock Exact steady periodic water waves with vorticity. Comm. Pure Appl. Math. 57 (2004), no. 4, 481--527.

\bibitem{ACWS2007}A.~Constantin, W.~Strauss,
\newblock Rotational steady water waves near stagnation. Philos. Trans. R. Soc. Lond. Ser. A Math. Phys. Eng. Sci. 365 (2007), no. 1858, 2227--2239.

\bibitem{RMCSWMHW2018}R.M.~Chen, S.~Walsh, M.H.~Wheeler,
\newblock Existence and qualitative theory for stratified solitary water waves. Ann. Inst. H. Poincar\'{e} C Anal. Non Lin\'{e}aire 35 (2018), no. 2, 517--576.

\bibitem{RMCSWMHW2022}R.M.~Chen, S.~Walsh, M.H.~Wheeler,
\newblock Center manifolds without a phase space for quasilinear problems in elasticity, biology, and hydrodynamics. Nonlinearity 35 (2022), no. 4, 1927--1985.

\bibitem{RMCSWMHW2024}R.M.~Chen, S.~Walsh, M.H.~Wheeler,
\newblock Global bifurcation for monotone fronts of elliptic equations. J. Eur. Math. Soc. (2024), (to appear).

\bibitem{DGV2022}A.~Doak, T.~Gao, J.M.~Vanden-Broeck,
\newblock Global bifurcation of capillary-gravity dark solitary waves on the surface of a conducting fluid under normal electric fields. Quart. J. Mech. Appl. Math. 75 (2022), no. 3, 215--234.

\bibitem{DAGTV2020}A.~Doak, T.~Gao, J.M.~Vanden-Broeck, J.J.S.~Kandola,
\newblock Capillary-gravity waves on the interface of two dielectric fluid layers under normal electric fields. Quart. J. Mech. Appl. Math. 73 (2020), no. 3, 231--250.

\bibitem{END1973}E.N.~Dancer,
\newblock Bifurcation theory for analytic operators. Proc. Lond. Math. Soc. 3 (1973), no. 26, 359--384.

\bibitem{DGTFZY2025}G.~Dai, T.~Feng, Y.~Zhang,
\newblock The Existence and Geometric Structure of Periodic Solutions to Rotational Electrohydrodynamic Waves Problem. J. Geom. Anal. 35 (2025) 179.

\bibitem{DGXFZY2024}G.~Dai, F.~Xu, Y.~Zhang,
\newblock The dynamics of periodic traveling interfacial electrohydrodynamic waves: bifurcation and secondary bifurcation. J. Nonlinear Sci. 34 (2024), no. 6, 99.

\bibitem{JDMDPMMMHW2024}J.~D\'{a}vila, M.~del Pino, M.~ Musso, M.H.~Wheeler,
\newblock Overhanging solitary water waves. Invent. Math. (2026), no. 1, 1--125.

\bibitem{JFDIGL1994}J.F.~De La Mora, I.G.~Loscertales,
\newblock The current emitted by highly conducting Taylor cones. J. Fluid Mech. 260 (1994) 155--184.


\bibitem{SADVMH2019}S.A.~Dyachenko, V.M.~Hur,
\newblock Stokes waves with constant vorticity: I. Numerical computation. Stud. Appl. Math. 142 (2019), no. 2, 162--189.

\bibitem{KOFDHH1954}K.O.~Friedrichs, D.H.~Hyers,
\newblock The existence of solitary waves. Comm. Pure Appl. Math. 7 (1954), 517--550.

\bibitem{MVFTGRRAD2022}M.V.~Flamarion, T.~Gao, R.~Ribeiro-Jr, A.~Doak,
\newblock Flow structure beneath periodic waves with constant vorticity under normal electric fields. Phys. Fluids, 34 (2022), no. 12, 127119.

\bibitem{MVFTGRRAD2023}M.V.~Flamarion, T.~Gao, R.~Ribeiro-Jr,
\newblock An investigation of the flow structure beneath solitary waves with constant vorticity on a conducting fluid under normal electric fields. Phys. Fluids, 35.3 (2023).

\bibitem{MVFEKRRNZ2024}M.V.~Flamarion, E.~Kochurin, R.~Ribeiro-Jr, N.~Zubarev,
\newblock Flow structure beneath periodic waves with constant vorticity under strong horizontal electric fields. Wave Motion, 131 (2024), 103413.

\bibitem{EMGSG2006} E.M.~Griffing, S.~George Bankoff, M.J.~Miksis, R.A.~Schluter,
\newblock Electrohydrodynamics of thin flowing films. J. Fluids Eng. 128 (2006) 276--283.

\bibitem{FG2025}F.~Gon\c{c}alves,
\newblock Gravity water waves over constant vorticity flows: From laminar flows to touching waves. Water Waves (2025), 1--19.

\bibitem{GHPHDTP2007}H.~Gleeson, P.~Hammerton, D.T.~Papageorgiou, J.M.~Vanden-Broeck,
\newblock A new application of the korteweg-de vries benjamin-ono equation in interfacial electrohydrodynamics. Phys. Fluids 19 (2007), no. 3.

\bibitem{TGZWDP2024}T.~Gao, Z.~Wang, D.~Papageorgiou,
\newblock Singularities of capillary-gravity waves on dielectric fluid under normal electric fields. SIAM J. Appl. Math. 84 (2024), no. 2, 523--542.

\bibitem{GWV2023}T.~Gao, Z.~Wang, J.M.~Vanden-Broeck,
\newblock Nonlinear wave interactions on the surface of a conducting fluid under vertical electric fields. Phys. D 446 (2023), 133651.

\bibitem{HMJ2013}M.J.~Hunt,
\newblock Linear and nonlinear free surface flows in electrohydrodynamics. PhD diss., UCL (University College London), 2013.

\bibitem{HMJDD2021}M.J.~Hunt, D.~Dutykh,
\newblock Free surface flows in electrohydrodynamics with a constant vorticity distribution. Water Waves 3 (2021), no. 2, 297--317.

\bibitem{SVHMHW2023}S.V.~Haziot, M.H.~Wheeler,
\newblock Large-amplitude steady solitary water waves with constant vorticity. Arch. Ration. Mech. Anal. 247 (2023), no. 2, Paper No. 27, 49 pp.

\bibitem{VMHMHW2020}V.M.~Hur, M.H.~Wheeler,
\newblock Exact free surfaces in constant vorticity flows. J. Fluid Mech. 896 (2020), R1, 10 pp.

\bibitem{VMHMHW2022}V.M.~Hur, M.H.~Wheeler,
\newblock Overhanging and touching waves in constant vorticity flows. J. Differential Equations 338 (2022), 572--590.


\bibitem{Jiang}Y.~Jiang, H.~Li, L.~Hua, D~.Zhang,
\newblock  Three-dimensional flow breakup characteristics of a circular jet with different nozzle geometries. Biosyst. Eng. 193 (2020) 216--231.


\bibitem{KK1982}K.~Kirchg\"{a}ssner,
\newblock Wave-solutions of reversible systems and applications. J. Differential Equations 45 (1982), no. 1, 113--127.

\bibitem{SFKPMS1997}S.F.~Kistler, P.M.~Schweizer,
\newblock Liquid film coating: scientific principles and their technological implications (1997).

\bibitem{Kozlov}V.~Kozlov, E.~Lokharu, M.H.~Wheeler,
\newblock Nonexistence of subcritical solitary waves. Arch. Ration. Mech. Anal. 241 (2021), no. 1, 535--552.

\bibitem{MAL1954}M.A.~Lavrentiev, I. On the theory of long waves. II. A contribution to the theory of long waves. Amer. Math. Soc. Translation 1954 (1954), no. 102, 53 pp.

\bibitem{LZW2017}Z.~Lin, Y.~Zhu, Z.~Wang,
\newblock Local bifurcation of electrohydrodynamic waves on a conducting fluid. Phys. Fluids 29 (2017), no. 3.

\bibitem{AM1986}A.~Mielke,
\newblock A reduction principle for nonautonomous systems in infinite-dimensional spaces. J. Differential Equations 65 (1986), no. 1, 68--88.

\bibitem{AM1988}A.~Mielke,
\newblock Reduction of quasilinear elliptic equations in cylindrical domains with applications, Math. Methods Appl. Sci. 10 (1988) 51--66.

\bibitem{DTP2019}D. T.~Papageorgiou,
\newblock Film flows in the presence of electric fields, Annual review of fluid mechanics 51 (2019), no. 1, 155--187.

\bibitem{JASPGS1985}J.A.~Simmen, P.G.~Saffman,
\newblock Steady deep-water waves on a linear shear current. Stud. Appl. Math. 73 (1985), no. 1, 35--57.

\bibitem{SVGVW2016}M.~Smit Vega Garcia, E.~V\u{a}rv\u{a}ruc\u{a}, G.S.~Weiss,
\newblock Singularities in axisymmetric free boundaries for electrohydrodynamic equations. Arch. Ration. Mech. Anal. 222 (2016), no. 2, 573--601.

\bibitem{AFTDHP1988}A.F.~Teles da Silva, D.H.~Peregrine,
\newblock Steep, steady surface waves on water of finite depth with constant vorticity. J. Fluid Mech. 195 (1988), 281--302.

\bibitem{JMVB1994}J.M.~Vanden-Broeck,
\newblock Steep solitary waves in water of finite depth with constant vorticity. J. Fluid Mech. 274 (1994), 339--348.

\bibitem{VBJM1995}J.M.~Vanden-Broeck,
\newblock New families of steep solitary waves in water of finite depth with constant vorticity. Eur. J. Mech. B Fluids 14 (1995), no. 6, 761--774.

\bibitem{VV2011}V.~Volpert,
\newblock Elliptic partial differential equations. Volume 1: Fredholm theory of elliptic problems in unbounded domains. Monographs in Mathematics, 101. Birkh\"{a}user/Springer Basel AG, Basel, (2011).

\bibitem{MHW2013}M.H.~Wheeler,
\newblock Large-amplitude solitary water waves with vorticity. SIAM J. Math. Anal. 45 (2013), no. 5, 2937--2994.

\bibitem{MHW20151}M.H.~Wheeler,
\newblock Solitary water waves of large amplitude generated by surface pressure. Arch. Ration. Mech. Anal. 218 (2015), no. 2, 1131--1187.

\bibitem{MHW20152}M.H.~Wheeler,
\newblock The Froude number for solitary water waves with vorticity. J. Fluid Mech. 768 (2015), 91--112.

\bibitem{SW2009}S.~Walsh,
\newblock Stratified steady periodic water waves. SIAM J. Math. Anal. 41 (2009), no. 3, 1054--1105.

\bibitem{WZ2017}Z.~Wang,
\newblock Modelling nonlinear electrohydrodynamic surface waves over three-dimensional conducting fluids. Proc. A. 473 (2017), no. 2200, 20160817, 20 pp.
\end{thebibliography}
\end{document}